\documentclass[pdflatex,sn-mathphys-num]{sn-jnl}

\usepackage{graphicx}%
\usepackage{multirow}%
\usepackage{amsmath,amssymb,amsfonts}%
\usepackage{amsthm}%
\usepackage{mathrsfs}%
\usepackage[title]{appendix}%
\usepackage{xcolor}%
\usepackage{textcomp}%
\usepackage{manyfoot}%
\usepackage{booktabs}%
\usepackage{algorithm}%
\usepackage{algorithmicx}%
\usepackage{algpseudocode}%
\usepackage{listings}%

\newtheorem{theorem}{Theorem}

\newtheorem{proposition}[theorem]{Proposition}%
\newtheorem{remark}{Remark}%

\newtheorem{structural}{Structural Condition}

\newtheorem{definition}{Definition}%

\newtheorem{lemma}[theorem]{Lemma}

\newtheorem{notation}{Notation}
\begin{document}

\title[Article Title]{Translating solitons to higher order mean curvature flows in Riemannian products}


\author[1]{\fnm{Jorge H. S. de Lira}}\email{jorge.lira@mat.ufc.br}

\author[2]{\fnm{Rafael R. de Farias}}\email{rafaelfariasmat@ufpa.br}
\equalcont{These authors contributed equally to this work.}

\affil[1]{\orgdiv{Federal University of Ceará}, \orgname{Campus do Pici}, \orgaddress{\street{Bloco 914}, \city{Fortaleza}, \country{Brazil}}}

\affil[2]{\orgdiv{Federal University of Pará}, \orgname{Campus do Guamá}, \orgaddress{\street{ICEN}, \city{Belém}, \country{Brazil}}}


\abstract{In this paper we prove existence and classification results for translating solitons defined as initial conditions for higher order mean curvature flows that are invariant by translations in warped product manifolds   $\mathbb{P}\times_\chi \mathbb{R}$. Here, $\mathbb P$ is a Cartan-Hadamard  manifold endowed with a rotationally symmetric metric and $\chi$ is a radial function defined in $\mathbb{P}$. In this setting, the higher order mean curvature flow is, up to a change of time parameter, given by translations along the factor $\mathbb{R}$ in the warped product. This setting encompasses the cases of translating solitons  in $\mathbb{R}^{n+1}$, $\mathbb{H}^n \times \mathbb{R}$ and $\mathbb{H}^{n+1}$ studied in recent papers. In particular we prove the existence of families of  bowl-type and catenoid-type translating solitons under mild assumptions about the curvature of the warped product. We also describe the asymptotic behavior for those solitons in terms of the geometry at infinity of $\mathbb{P}$. Our assumptions about the ambient metric allow us to control the higher order mean curvature of cylinders and to use them as barriers.}




\maketitle

\section{Introduction}\label{sec1}

We consider product manifolds $\bar M^{n+1} = \mathbb{P}\times \mathbb{R}$ where $(\mathbb{P}, \sigma)$ is a $n$-dimensional complete Riemannian manifold. We suppose that the Riemannian metric $\bar g$ in $\bar M$ is a warped metric of the form 
\begin{equation}
\label{warped-2}
 \bar g = \chi^2 ds^2 + \sigma,   
\end{equation}
for some smooth positive function $\chi: \mathbb P\to \mathbb R$. Here, $s$ is the natural coordinate in the factor $\mathbb{R}$ of the product $\bar M$. The vector field $X=\partial_s$ is a Killing vector field with norm given by $|X|=\chi$. If $\chi$ is constant, $X$ is a parallel vector field. We indicate \eqref{warped-2} by $\bar M =\mathbb{P}\times_{\chi} \mathbb{R}$.
This geometric setting encompasses space forms as $\mathbb{R}^{n+1}$ and $\mathbb{H}^{n+1}(\kappa)$ as well as Riemannian products as $\mathbb{H}^n (\kappa) \times \mathbb{R}$ among other examples.

The main results in this paper concern existence of hypersurfaces $M$ immersed into  $\bar M$ which are initial conditions for  solutions of a higher order mean curvature flow invariant by the translations generated by $X$. Those geometric flows are defined as a one-parameter family of hypersurfaces $\Psi:[0, t^*) \times M \to \bar M$, for some $t^*>0$, satisfying
\begin{equation}
\label{flow}
 (\partial_t \Psi)^\perp = S^\alpha N.
 \end{equation}
Given $m\in\{2,...,n-1,n\}$, the speed of this flow is a power, with exponent $\alpha\in\{1/m,1\}$, of the $m$-curvature of $M$  defined as
\begin{equation}
\label{scalar}
    S = \underset{i_1 < \ldots < i_m}{\sum_{\{i_1, \ldots, i_m\}\subset\{1,...,n\}}}  \lambda_{i_1} (A)\cdot \ldots \cdot \lambda_{i_m} (A).
\end{equation}
The case $m=2$ corresponds to the extrinsic scalar curvature of the immersed hypersurface and $S$ is the Gauss-Kronecker curvature of $M$ when $m=n$. 

The Weingarten map $A$ is computed with respect to the unit normal vector field $N=N|_{\Psi_t(M)}$ in \eqref{flow} which defines an orientation for  $\Psi_t(M)$, for each $t\in [0, t^*)$, where $\Psi_t(\cdot) = \Psi(t, \,\cdot)$. The principal curvatures  $\lambda_i(A)$ in \eqref{scalar} are, by definition, the eigenvalues of $A$. Note that the left-hand side in \eqref{flow}
involves only the orthogonal projection of the variational vector field $\partial_t \Psi$ onto the normal bundle. This means that the flow is defined up to local tangential diffeomorphisms in $M$. In section \ref{prelim} one proves that the condition of invariance by translations has the following infinitesimal expression:
\begin{equation}
\label{sol-eq-tensor}
   II_{-S^\alpha N} + \frac{c}{2} \pounds_{X^T} g = 0,
\end{equation}
where $g$ and $II$ are, respectively, the induced metric and second fundamental form of $\Psi_t (M)$ and $c$ is a constant related to the ratio between the parameters $s$ of the flow of $X$ and the time parameter in \eqref{flow}.
Taking traces in \eqref{sol-eq-tensor} one obtains the higher order mean curvature soliton equation
\begin{equation}
    \label{sol-eq-scalar}
    S^\alpha = c \langle X, N\rangle.
\end{equation}
In what follows  we refer to \eqref{sol-eq-scalar} as the $m$-th mean curvature flow soliton equation. In the Euclidean setting this soliton equation for $\alpha = 1/m$ is invariant by rescaling of the metric.

Our main existence results (see Theorems \ref{ExistenceTheorem-m<nXnonparallel} and \ref{ExistenceTheorem-m=nXparallel})  are obtained under the assumptions that $\mathbb{P}$ is a Hadamard manifold with a rotationally invariant metric
\begin{equation}
    \sigma = dr^2 + \xi^{2}(r) g_{\mathbb S^{n-1}}
\end{equation}
and that $\chi = \chi(r)$, where $r$ is the Riemannian distance in $\mathbb{P}$ from some pole $o\in \mathbb{P}$. Some structural assumptions concerning the behavior of the warping functions $\chi$ and $\xi$ are required to establish the existence of complete examples of solitons. At this respect, we refer the reader to section \ref{Section3}. Theorems \ref{ExistenceTheorem-m<nXnonparallel} and \ref{ExistenceTheorem-m=nXparallel} guarantee the existence of $m$-th mean curvature flow solitons in $\mathbb{R}^{n+1}$, $\mathbb{H}^n \times \mathbb{R}$ and $\mathbb{H}^{n+1}$ since the structural assumptions are valid in these particular cases.

In this setting, we have proved the existence of i) $C^2$ bowl-type solitons, ii) a one-parameter family of complete and $C^1$  solitons to the $m$-th mean curvature flow for $m$ odd, and iii) two non-complete one-parameter families of $C^1$ solitons to the $m$-th mean curvature flow for $m$ even. The existence of such examples corroborates the results in \cite{Ronaldo}, where the authors consider the particular case when $X$ is parallel, specifically the Riemannian products  $\mathbb{R}^{n+1}=\mathbb{R}^n\times \mathbb{R}$ and $\mathbb{H}^n\times\mathbb{R}$.

The main technical strategy in the proof of the existence results consists of using the structural conditions to control the behavior of the $m$-th mean curvature $S_m(r)$ of the Killing cylinders of radii $r>0$ in $\mathbb{P}\times \mathbb{R}$. The Gauss-Kronecker curvature case, $m=n$, is distinct from the others, therefore we consider the cases $m\in \{2,\dots,n-1\}$ and $m=n$ separately. In fact, we analyze the cases  $\alpha=1/m$ and $m$ even; $\alpha=1$ and $m$ even; and $\alpha\in\{1/m,1\}$ and $m$ odd separately. Moreover, for $m=n$ we have different approaches for the cases when $X$ is  parallel, where $\chi\equiv1$ and $S_n(r)\equiv0$; and the cases when $\chi^\prime(r)>0$.

The structure of this paper is as follows. In Section \ref{prelim} we introduce the notions and fundamental equations of $m$-th mean curvature flows invariant by translations and the corresponding translating solitons for those flows. In Sections \ref{Section3} and \ref{Section4}, we model the $m$-th mean curvature soliton equation for rotationally invariant  Killing graphs in terms of a system of first-order differential equations, whose phase portrait analysis is done in Section \ref{Section4} following the techniques introduced in in \cite{Bueno}. Also, we describe the asymptotic behavior of the rotating flow solitons for either $m<n$ or $m=n$ and nonparallel $X$. Finally, in Section \ref{Section5} we give some explicit expressions of some examples of solitons. 

Mean curvature flow solitons have been a major topic of research with a massive body of papers devoted to it. We refer the reader to the encyclopedic approach to the subject in \cite{andrews}. The study of solitons for extrinsic non-linear geometric flows is relatively more recent and some references which were fundamental for our contributions here are \cite{torres}, \cite{Ronaldo}, \cite{Lima2025}, \cite{Pipoli2021}, \cite{Lira2019}, \cite{Lima2023} and \cite{Pipoli2020}.

The research leading to these results received funding from CNPq under Grant Agreement No 306626/2022-5 and from the Coordenação de Aperfeiçoamento de Pessoal de Nível Superior – Brasil (CAPES) – Finance Code 001. All figures were created using Mathematica software.

\section{Preliminares}
\label{prelim}

We consider the variation of a given isometric immersion $\psi: M^n \to \bar M^{n+1}$ defined by the geometric flow $\Psi:(\omega_*, \omega^*)\times M \to \bar
M$ whose speed is the function $F$ of the Weingarten map (more precisely, of its principal values) given by
\begin{equation}
F(h_{ij}) = F({\pmb \kappa}) \doteq S^\alpha,
\end{equation}
where $\alpha \in \{1/m, 1\}$ and ${\pmb \kappa}$ is the vector of the principal curvatures of the Weingarten map whose components are $h_{ij}$. This means that
\begin{align}
    \left( \frac{\partial\Psi}{\partial t} \right)^\perp= S^\alpha N.
\end{align}
Let $\Phi:\mathbb{R}
\times \bar M \to \bar M$  be the
flow generated by $X$ defined in $\mathbb{R}$.
Let  $s$ be the flow parameter in $\Phi$ and define
\begin{equation}
\widetilde \Psi_{t} (x) =\widetilde \Psi(t, x)=
\Phi^{-1}(\sigma(t),\Psi_{t}(x)), \quad x\in M,
\end{equation}
where $\sigma:(\omega_*,\omega^*)\to \mathbb{R}$
is a
reparametrization of the flow lines of $X$ of the form
\[
s=\sigma(t).
\]
Equivalently we can write
\begin{equation}
\label{psiphi} \Psi(t,x) =
\Phi(\sigma(t),\widetilde\Psi(t,x)), \quad (t,x)\in (\omega_*,\omega^*)\times M .
\end{equation}

\begin{definition}\label{def1}
\label{selfsimilar} Let $\bar M^{n+1}$ be a Riemannian manifold
endowed with a Killing vector field $X\in \Gamma(T\bar
M)$. Given an $n$-dimensional Riemannian manifold $M^n$, we say that
the higher order mean curvature flow $\Psi:(\omega_*, \omega^*)\times M \to \bar
M$ is \textit{invariant by translations} generated by $X$ if  there exists an isometric immersion
$\psi:M \to \bar M$ and a reparametrization $\sigma:(\omega_*, \omega^*)\to \mathbb{R}$ of the flow lines of $X$ such that
\begin{equation}
\Psi_{t}(M) =\Phi_{\sigma(t)}(\psi(M)),
\end{equation}
for all $t\in (\omega_*,\omega^*)$, where
 $\Phi:\mathbb{R}\times \bar M\to\bar M$ is the flow generated by $X$. In other terms, $\widetilde\Psi_t(M) = \psi(M)$, for all $t\in (\omega_*, \omega^*)$.
\end{definition}

\begin{remark}
Although  Definition \ref{def1} does not require in principle any special properties of $X$, we will restrict ourselves to Killing fields. 
\end{remark}
\noindent Recall  that $X$ is said to be Killing if the equation
\begin{equation}
\label{killing-1} \pounds_{X}\bar  g = 0
\end{equation}
holds.
In this case, it turns out that each map $\Phi_s = \Phi(s, \cdot):\bar M \to \bar M$, $s\in \mathbb{R}$, is an isometry  in the sense that
\begin{equation}
\label{conformal} \Phi^*_s \bar g =  \bar g.
\end{equation}
We are tacitly assuming that there are no singular points of $X$ in
$\bar M$ by replacing $\bar M$ with a proper open subset of it,
if necessary. Considering  $\mathbb{P}$ as a fixed integral leaf of the distribution orthogonal to $X$, it is convenient to parameterize the flow $\Phi$ by fixing initial conditions on  $\mathbb{P}$. In other words, this means that we consider $\Phi:\mathbb{R}\times \mathbb{P}\to \bar M$ as
a global chart of $\bar M$.  Having fixed this map,  the integral  leaves $\mathbb{P}_s:=\Phi_s(\mathbb{P})$ are identified with the slices $\{s\}\times \mathbb{P}$, $s\in \mathbb{R}$. Recall that we can describe $\bar M$ as the \textit{warped} product $\mathbb{P} \times_\chi \mathbb{R}$ with warped
Riemannian metric given by
\[
\chi^2 \textrm{d}s^2 +  \sigma,
\]
where $\sigma$ is the metric in $\mathbb{P}$ and $\chi = |X|$ is constant along flow lines of $X$.  In terms of this model of $\bar M$ we have
\begin{equation}
\label{conformal-warped}
X = \partial_s .
\end{equation}

Once established this geometric setting and notations, we may deduce some fundamental  equations describing  translating solitons in analytical terms. We start proving some preliminary facts about the interaction between Killing fields and $m$-th mean curvature flows and solitons.

\begin{lemma}
    If $U\in\Gamma(T\bar M)$ then
    \begin{equation}
    \label{Xderivative-killing}
         \bar \nabla_U X = \langle U, \bar \nabla \log \chi \rangle X - \langle U, X \rangle \bar \nabla \log \chi .
\end{equation}
\end{lemma}
\noindent \textit{Proof}. Given $\textrm{u},\textrm{v}\in\Gamma(T\bar M)$, one has
\begin{align*}
    & \langle \bar\nabla_{{\textrm{u}}} X, {\textrm{v}}\rangle =\chi^{-4}\langle {\textrm{u}}, X\rangle \langle {\textrm{v}}, X\rangle \langle \bar\nabla_{X} X, X\rangle \\
    & \ +\langle \bar\nabla_{{\textrm{u}}^\perp} X, {\textrm{v}}^\perp\rangle + \chi^{-2}\langle {\textrm{u}}, X\rangle \langle \bar\nabla_{X} X, {\textrm{v}}^\perp\rangle + \chi^{-2}\langle {\textrm{v}}, X\rangle \langle \bar\nabla_{{\textrm{u}}^\perp} X, X\rangle.
\end{align*}
Here and in what follows, $\bar\nabla$ denotes the Riemannian connection in $(\bar M, \bar g)$. Since $\langle \bar\nabla_X X, V \rangle = - \langle \bar\nabla_V X, X \rangle = -(1/2)\langle \bar\nabla \chi^2 , V \rangle$ we have that
\[
\bar\nabla_X X = -\frac{1}{2}\bar\nabla \chi^2 \quad \mbox{ and } \quad \langle\bar\nabla_{{\textrm{u}}^\perp} X, X\rangle = \frac{1}{2}\langle \bar\nabla \chi^2, {\textrm{u}}^\perp\rangle.
\]
Using the fact that the leaves perpendicular to $X$ are totally geodesic and that $\chi^2 = |X|^2$ is constant along the flow lines of $X$, 
one concludes that
\begin{align*}
    & \langle \bar\nabla_{{\textrm{u}}} X, {\textrm{v}}\rangle =-\langle {\textrm{u}}, X\rangle \langle \bar\nabla\log\chi, {\textrm{v}}\rangle  + \langle {\textrm{v}}, X\rangle \langle \bar\nabla\log\chi, {\textrm{u}}\rangle.
\end{align*}
This finishes the proof.
\hfill $\square$

\begin{proposition}
\label{soliton-prop} Let $\Psi:(\omega_*, \omega^*)\times M \to
\bar M$ be a $m$-th mean curvature flow invariant by translations generated by the Killing  vector field $X\in \Gamma(T\bar M)$. Then for all
$t\in (\omega_*, \omega^*)$ there exists a constant $c_t$
such that
\begin{equation}
\label{solitonA} c_t X =c_t\Psi_{t *}T + S^\alpha N,
\end{equation}
where $S$ is the $m$-th mean curvature \eqref{scalar} of $\Psi_t= \Psi(t, \cdot\,)$ and
$T\in\Gamma(TM)$ is the pull-back by $\Psi_t$ of the tangential
component of $X$. 
Moreover, 
\begin{equation}
\label{solitonB} II_{-S^\alpha N} + \frac{c_t}{2}\pounds_T g =
0,
\end{equation}
where $g$ is the metric induced in $M$ by $\Psi_t$ and $II_{-S^\alpha N}$ is its second fundamental form in the direction of $-S^\alpha N$.
\end{proposition}

\noindent \textit{Proof.} Differentiating both sides in
(\ref{psiphi}) with respect to $t$ we obtain
\begin{eqnarray}
\frac{d\Psi}{dt}\Big|_{(t,x)}&=&
\frac{\partial\Phi}{\partial s} \Big|_{(\sigma
(t),\widetilde\Psi(t, x))}\frac{d\sigma}{dt}\Big|_{t}
+\Phi_{\sigma(t)*}(\widetilde\Psi(t, x))\frac{d\widetilde\Psi}{dt}\Big|_{(t, x)}\nonumber\\
& = & X(\Psi(t,
x))\frac{d\sigma}{dt}\Big|_{t}+\Phi_{\sigma(t)*}(\widetilde\Psi(t,
x))\frac{d\widetilde\Psi}{dt}\Big|_{(t, x)},\label{XT}
\end{eqnarray}
where $\Phi_\sigma = \Phi(\sigma, \cdot).$ Since $\Psi$ is a higher order mean curvature flow invariant by translations generated by $X$, there exists an isometric
immersion $\psi:M\to \bar M$ such that $\Psi(0,\,\cdot\,)=\psi$
and
$\widetilde\Psi_{t} (M) = \psi(M)$
for all $t\in (\omega_*,\omega^*)$.  This 
implies that
\[
\Phi_{\sigma(t)*}(\widetilde\Psi(t,
x))\frac{d\widetilde\Psi}{dt}\Big|_{(t, x)} \in T_{\Psi
(t, x)} \Psi_t(M).
\]
We conclude that for all $t\in (\omega_*,\omega^*)$ the
tangential component of $\frac{d\sigma}{dt} X$ onto
$\Psi_t(M)$ is given by
\[
X^\top(\Psi(t,x))\frac{d\sigma}{dt}\Big|_{t} =
-\Phi_{\sigma(t)*}(\widetilde\Psi(t,x))\frac{d\widetilde\Psi}{dt}\Big|_{(t,x)}
\]
where the superscript $\top$ denotes tangential projection.
We note that the  expression
\begin{equation}
c_t \Psi_{t*}T(t,x) =
-\Phi_{\sigma(t)*}(\widetilde\Psi(t,x))\frac{d\widetilde\Psi}{dt}\Big|_{(t,x)}.
\end{equation}
defines a vector field $T(t,\cdot)\in \Gamma(TM)$,
for each $t\in (\omega_*,\omega^*)$, where $c_t = \frac{d\sigma}{dt}\big|_t$. We conclude from (\ref{XT}) that
\begin{equation}
\label{soliton-quase}
S^\alpha N|_{\Psi(t,x)} =
c_t X(\Psi(t,x))
-c_t\Psi_{t*}T(t,x).
\end{equation}
Then, we rewrite (\ref{soliton-quase}) in the form
\begin{equation}
\label{soliton-tau}
c_t X|_{\Psi_t} = 
c_t\Psi_{t *}T_t   +  S^\alpha N|_{\Psi_\tau} 
\end{equation}
where $T_t(x)= T(t,x)$.

Next, for a fixed $t$,  one denotes by  $g$ and $\nabla$, respectively, the induced metric and connection of the immersion
$\Psi_t$. Hence,  it follows from  (\ref{soliton-tau}) that
\begin{equation}
\label{nablaX} c_t\bar\nabla_{\Psi_* U} X = c_t\Psi_{*}\nabla_U T
+ c_t (\bar\nabla_{\Psi_* U} \Psi_*T)^\perp +\bar\nabla_{\Psi_*U}
S^\alpha N.
\end{equation}
Taking the normal projection in both sides one has
\begin{equation*}
\label{nablaX-bis} c_t(\bar\nabla_{\Psi_* U} X)^\perp = c_t (\bar\nabla_{\Psi_* U} \Psi_*T)^\perp +\bar\nabla^\perp_{\Psi_*U} S^\alpha N.
\end{equation*}
Since $X$ is a Killing vector field we have by \eqref{Xderivative-killing} that
\begin{equation}
    \label{killing-eq}
    \bar\nabla_{\Psi_* U} X = - \langle \Psi_*U , X\rangle \bar\nabla \log \chi + \langle \bar\nabla \log\chi, \Psi_* U\rangle X,
\end{equation}
where $\chi = |X|$. Hence, it follows from  \eqref{nablaX} that
\begin{align*}
   &  c_t (\bar\nabla_{\Psi_* U} \Psi_*T)^\perp +\bar\nabla^\perp_{\Psi_*U} S^\alpha N = -c_t \langle \Psi_*U , X\rangle \bar\nabla^\perp \log \chi + c_t \langle \bar\nabla \log\chi, \Psi_* U\rangle X^\perp\\
    & \,\, =-c_t \langle \Psi_*U , X\rangle \bar\nabla^\perp \log \chi +  \langle \bar\nabla \log\chi, \Psi_* U\rangle S^\alpha N .
\end{align*}
Moreover, combining \eqref{nablaX} and \eqref{killing-eq}
 one has
 \begin{align*}
     & 2 II_{-S^\alpha N} (U, V) + c_t \pounds_{T} g (U,V) = c_t \langle \bar\nabla_{\Psi_* U} X, \Psi_* V\rangle + c_t \langle \bar\nabla_{\Psi_* V} X, \Psi_* U\rangle\\
     & \,\, = - \langle \Psi_*U , X\rangle \langle\bar\nabla \log \chi, \Psi_* V\rangle + \langle \bar\nabla \log\chi, \Psi_* U\rangle \langle X, \Psi_* V\rangle\\
     & \,\,\,\, - \langle \Psi_*V , X\rangle \langle\bar\nabla \log \chi, \Psi_* U\rangle + \langle \bar\nabla \log\chi, \Psi_* V\rangle \langle X, \Psi_* U\rangle =0.
 \end{align*}
 Since the Killing equation implies that $X$ is divergence-free, this completes the proof of Proposition \ref{soliton-prop}. 
 
 \hfill $\square$

Motivated by the above geometric setting, we define a general notion of higher order mean curvature flow soliton  with respect to a given Killing vector field $X\in \Gamma(T\bar M)$ as in Definition \ref{DefinitionM-meanCurvatureFlow}.

\begin{definition}
\label{DefinitionM-meanCurvatureFlow} Let $M^n$ be a $n$-dimensional Riemannian manifold. An isometric immersion   $\psi: M\to \bar M$ oriented by a unit normal vector field is a \textit{$m$-th mean curvature flow soliton} with respect to the Killing field $X$ in $\bar M$ if its $m$-th mean curvature $S$  satisfies the equation
\begin{equation}\label{soliton-translating}
     S^\alpha = c \langle X, N\rangle,    \end{equation}
     for some constant $c>0$ and $\alpha \in \{1, 1/m\}$. We alternatively refer to $m$-th mean curvature flow solitons in this context as {\bf translating solitons}. 
\end{definition}

With a slight abuse of notation, we also say that
the hypersurface  $\psi(M)$ itself is a soliton to the $m$-th mean curvature flow {\textrm{(}}with respect to the vector field $X${\textrm{)}}.

We observe that equation \eqref{solitonA} is enough to deduce the following important consequences that we have considered in Proposition \ref{soliton-prop} in the context of extrinsic geometric flows.

\begin{proposition}
\label{soliton-prop-2} Let $\psi:M^n\to \bar M^{n+1}$ be a soliton to the $m$-th mean curvature flow with respect to a Killing vector field  $X\in \Gamma(T\bar M)$. Then along $\psi$ we have
\begin{equation}
\label{solitonB-2} II_{-S^\alpha N} + \frac{c}{2}\pounds_{T} g =
0,
\end{equation}
where $g$ is the metric induced in $M$ by $\psi$ and $II_{-S^\alpha N}$
is its second fundamental form of $\psi$ in the direction of $-S^\alpha N$. Here   the vector field $T$ is defined by $\psi_*T = X^\top$.
Furthermore,
\begin{equation}
\label{solitonC-3} \langle \nabla S^\alpha, \,\cdot\rangle N+c\, II(T,\,\cdot)=-c \langle T, \,\cdot\rangle \bar\nabla^\perp \log \chi +  \langle \nabla \log\chi, \,\cdot \rangle S^\alpha N,
\end{equation}
where $II$ is the second fundamental tensor of $\psi$ and $\nabla$ is the Riemannian connection in $M$ induced by $\psi$.
\end{proposition}

\noindent \textit{Proof.}  Using \eqref{solitonA} by a direct computation we have for any tangent vector fields $U, V\in \Gamma(TM)$
\begin{eqnarray*}
& & c\langle \bar\nabla_{\psi_* U} X, \psi_* V\rangle+c\langle \bar\nabla_{\psi_* V} X, \psi_* U\rangle\\
& &\,\, = c\langle \bar\nabla_{\psi_* U} \psi_* T, \psi_* V\rangle+c\langle \bar\nabla_{\psi_* V} \psi_* T, \psi_* U\rangle + \langle \bar\nabla_{\psi_* U} S^\alpha N, \psi_* V\rangle+\langle \bar\nabla_{\psi_* V} S^\alpha N, \psi_* U\rangle\\
& &\,\, = c\langle \nabla_{U} T, V\rangle + c\langle \nabla_{V} T, U\rangle+ 2 II_{-S^\alpha N} (U,V).
\end{eqnarray*}
Since $X$ is a Killing field it follows from \eqref{killing-eq} that
\[
c\langle \bar\nabla_{\psi_* U} X, \psi_* V\rangle+c\langle \bar\nabla_{\psi_* V} X, \psi_* U\rangle = 0,
\]
from which one concludes that
\[
II_{-S^\alpha N} +\frac{c}{2}\pounds_{T} g = 0.
\]
Using the fact that $X$ is Killing, one gets
\begin{eqnarray*}
&  -c \langle \Psi_*U , X\rangle \bar\nabla^\perp \log \chi +  \langle \bar\nabla \log\chi, \Psi_* U\rangle S^\alpha N  = c(\bar\nabla_{\psi_* U} X)^\perp  \\
& = c (\bar\nabla_{\psi_*U}\psi_*T)^\perp + c(\bar\nabla_{\psi_* U} X^\perp)^\perp\\
&  \,\, = c\,II (T, U) + (\bar\nabla_{\psi_*U} (S^\alpha N)^\perp = c\, II (T, U) + \langle \nabla S^\alpha, U\rangle N,
\end{eqnarray*}
what concludes the proof of Proposition 
\ref{soliton-prop-2}. \hfill $\square$

\section{Higher order mean curvature of rotationally invariant hypersurfaces}
\label{Section3}

In this section, we deduce some analytical formulations for $m$-th mean curvature flow solitons in the case when  $\mathbb{P}$ is a Hadamard manifold and its Riemannian  metric $\sigma$ is rotationally invariant in the sense that there exists a pole $o\in \mathbb{P}$ and Gaussian global coordinates $(r, \theta)\in \mathbb{R}^+ \times \mathbb{S}^{n-1}$ centered at $o$ in terms of which $\sigma$ may be written as
\begin{align}
    \sigma = dr^2 + \xi^2(r) d\theta^2,
\end{align}
where $d\theta^2$ is the standard metric in $\mathbb{S}^{n-1}$ and $\xi$ is a smooth positive function that extends smoothly as $r\to 0^+$. We also suppose that $\chi= |X|$ depends (smoothly) only on $r=\operatorname{dist}_P (o, x)$, that is, $\chi(x) = \chi(r(x))$. In sum, the metric $\bar g$ in $\bar M = \mathbb{P} \times_\chi \mathbb{R}$ has a doubly warped structure described in terms of the ``cylindrical'' coordinates $(r, \theta, s)$ as
\begin{equation}
\label{riem-amb-met}
    \bar g = dr^2 + \xi^2(r) d\theta^2 + \chi^2(r) ds^2.
\end{equation}

From now on,  we assume the validity of the following  structural assumptions about $\chi$ and $\xi$.

\begin{structural}
\label{StructuralConditionChi1} 
The positive function $\chi(r)$ is  non-decreasing; either $\chi$ is constant {\rm(}that could be fixed as identically $1$ up to some rescaling in $\mathbb{R}${\rm)} or $\chi(r) \to  \chi_\infty\in (1,+\infty)$ or $\chi(r)\to +\infty$ as $r\to \infty$.
\end{structural}

\begin{structural}
\label{StructuralConditionxi2}
The function  $\xi'(r)/\xi(r)$ is decreasing and the function $(\xi'(r)\chi'(r))/(\xi(r)\chi(r))$ is non-increasing.
\end{structural}

\begin{structural}
 \label{StructuralConditionrto0}
As $r\to0$, the functions $\xi$ and $\chi$ behave  as 
        \begin{align*}
            \xi(r)=r+O(r^{3}), \\
            \chi(r)\equiv 1 \mbox{ or }\chi(r)=1+\frac{r^2}{2}+O(r^4).
        \end{align*}
\end{structural}

\begin{structural}
\label{StructuralConditionrtoinfty}
The functions $\xi$ and $\chi$ behave asymptotically {\rm(}as $r\to+\infty${\rm)} according to one of the following models: 
        \begin{itemize}
            \item{Euclidean model $\mathbb{R}^{n+1}$}, when
            \begin{align}
                \xi(r) = r + O(r^{-1}),  \ \ \ \chi(r)=\chi_\infty + O(r^{-1}), \quad \mbox{ for some }\quad \chi_\infty \in [1,+\infty).
            \end{align}
\item{Riemannian  product $\mathbb{H}^{n}\times\mathbb{R}$ model}, when
            \begin{align}
                \xi(r)=e^r+O(r^{-1}), \ \ \ \chi(r) = \chi_\infty + O(r^{-1}), \quad \mbox{ for some }\quad \chi_\infty \in [1,+\infty).
            \end{align}
            \item{Hyperbolic model $\mathbb{H}^{n+1}$}, when
            \begin{align}
                \xi(r), \,\chi(r)=e^r+O(r^{-1}).
            \end{align}
        \end{itemize}
\end{structural}


\begin{definition}
    \label{notation}
    If the warping functions $\xi$ and $\chi$ behaves asymptotically, as $r\to+\infty$, according to the Euclidean, Riemannian product  or hyperbolic models, we say that the metric asymptotes Euclidean, Riemannian product or hyperbolic models, respectively. If the metric asymptotes either Euclidean model or Riemannian product model, we say that $X$ is asymptotically parallel.
\end{definition}

Recall that $X=\partial_s$ is Killing vector field (in fact, a parallel vector field if $\chi$ is constant) with  
\[
\bar\nabla_{\partial_r} X = \frac{\partial_r \chi}{\chi} X = \frac{\chi'(r)}{\chi(r)} X \quad \mbox{ and }\quad
\bar\nabla_X X  =  -\chi(r) \chi'(r) \partial_r.
\]
Moreover we have
\[
\langle \bar\nabla_{\partial_{\theta^i}} X, X\rangle = \langle \bar\nabla_{\partial_{\theta^i}} X, \partial_r\rangle = \langle \bar\nabla_{\partial_{\theta^i}} X, \partial_{\theta^j}\rangle = 0,
\]
\[
\bar\nabla_{\partial_r} \partial_r = 0 \quad \mbox{ and} \quad \bar\nabla_{\partial_{\theta^i}} \partial_r = \frac{\xi'(r)}{\xi(r)}\partial_{\theta^i}.
\]
A rotationally invariant hypersurface $M$ in $\bar M$ may be parameterized in clylindrical coordinates as 
\[
\Psi(\tau, \vartheta) = (r(\tau), \vartheta, s(\tau)), 
\]
for some parameter $\tau$. Hence the tangent space at a given point at $M$ is spanned by the coordinate vector field 
\[
\partial_\tau \Psi = \dot r \partial_r + \dot s \partial_s
\quad \mbox{ and } \quad
\partial_{\theta^i}\Psi = \partial_{\theta^i}|_{\Psi},
\]
where the superscript $\pmb{^\cdot}$ denotes derivatives with respect to the parameter $\tau$. The induced metric in $M$ has local components in terms of the coordinates $(s, \theta^1, \ldots, \theta^{n-1})$ given by
\begin{align*}
     g_{\tau\tau} = \dot r^2 + \chi^2(r) \dot s^2 =: W^2,\quad 
      g_{\tau i} = 0,\quad
     g_{ij} = \xi^2(r) \theta_{ij},
\end{align*}
where $\theta_{ij}$ are the local components of $d\theta^2$. The  components of the contravariant metric are
\begin{align*}
     g^{\tau\tau} = W^{-2},\quad g^{\tau i} = 0,\quad
      g^{ij} = \xi^{-2}(r) \theta^{ij}.
\end{align*}
An orientation for $M$ is given by
\begin{align}
\label{N-OrientationFormula}
N = \frac{1}{\chi W}\big(-\chi^2(r) \dot s \partial_r + \dot r \partial_s \big),
\end{align}
Now we compute
\begin{align*}
    & \bar\nabla_{\partial_\tau\Psi} (\chi W N) = -\frac{d}{d\tau} (\chi^2 \dot s) \partial_r + \frac{d}{d\tau}\dot r \partial_s - \chi^2 \dot s \bar\nabla_{\partial_\tau \Psi} \partial_r + \dot r \bar\nabla_{\partial_\tau\Psi} \partial_s\\
    & = -\frac{d}{d\tau} (\chi^2 \dot s) \partial_r + \frac{d}{d\tau}\dot r \partial_s + \big(- \chi^2 \dot s^2 \bar\nabla_{\partial s}\partial_r + \dot r^2 \bar\nabla_{\partial_r}\partial_s + \dot r \dot s \bar\nabla_{\partial_s}\partial_s \big)\\
    & = -\frac{d}{d\tau} (\chi^2 \dot s) \partial_r + \frac{d}{d\tau}\dot r \partial_s + \big(\dot r^2 - \chi^2 \dot s^2\big) \frac{\chi'(r)}{\chi(r)}\partial_s  - \dot r \dot s \chi(r) \chi'(r) \partial_r.
\end{align*}
Therefore the component of the second fundamental form of $M$ in the direction of $\partial_\tau \Psi$ is given by
\begin{align*}
  &  h_{\tau\tau} = -\langle \bar\nabla_{\partial_\tau \Psi} N, \partial_\tau \Psi\rangle \\
  & \,\, = \frac{1}{\chi W} \Big(\dot r \frac{d}{d\tau} (\chi^2 \dot s) -\chi^2 \dot s \frac{d}{d\tau}\dot r +  \dot r^2 \dot s \chi(r) \chi'(r) - \dot s\big(\dot r^2 - \chi^2 \dot s^2\big) \chi(r)\chi'(r)  \Big).
\end{align*}
Hence, we have
\begin{align}
    h_{\tau\tau} = \frac{1}{\chi W} \Big(\dot r \frac{d}{d\tau} (\chi^2 \dot s) -\chi^2 \dot s \frac{d}{d\tau}\dot r +  \chi'\chi^3 \dot s^3  \Big).
\end{align}
Now we compute
\begin{align*}
\bar\nabla_{\partial_{\theta^i} \Psi}(\chi W N) = \bar\nabla_{\partial_{\theta^i}} \big(-\chi^2 \dot s \partial_r + \dot r \partial_s\big) = -\chi^2 \dot s \bar\nabla_{\partial_{\theta^i}}\partial_r =  -\chi^2 \dot s \frac{\xi'(r)}{\xi(r)} \partial_{\theta^i}  .
\end{align*}
Therefore 
\begin{align}
    h_{\tau i} \doteq -\langle \bar\nabla_{\partial_{\theta^i}\Psi} N, \partial_\tau \rangle= 0
\end{align}
and
\begin{align}
    h_{ij} \doteq  -\langle \bar\nabla_{\partial_{\theta^i}\Psi} N, \partial_{\theta^j} \rangle = \frac{1}{\chi W}\chi^2 \dot s \frac{\xi'(r)}{\xi(r)}  g_{ij}
    = \frac{\chi}{W} \dot s \xi \xi^\prime \theta_{ij}.
\end{align}
Therefore the principal values of  the Weingarten map
$h^a_b  = g^{ac} h_{cb}$
of the isometric immersion of $M$ into $\bar M$ are given by
\begin{align}
    \label{k}
    \kappa_\tau \doteq h^\tau_\tau 
    = g^{\tau \tau} h_{\tau \tau} + g^{\tau i} h_{i \tau}
    =g^{\tau \tau} h_{\tau \tau}
    = \frac{1}{\chi W^3} \Big(\dot r \frac{d}{d\tau} (\chi^2 \dot s) -\chi^2 \dot s \frac{d}{d\tau}\dot r +  \chi'\chi^3 \dot s^3 \Big),
\end{align}
in $\partial_r$ direction, and
\begin{align}
\label{kappa}
    \kappa_\theta \doteq  \frac{1}{W}\chi \dot s \frac{\xi'(r)}{\xi(r)},
\end{align}
in any direction of $\mathbb{S}^{n-1}$.  Therefore the $m$-th mean curvature of $M$ is given by
\begin{align*}
    S= {n-1\choose m-1}\kappa_r \kappa_\theta^{m-1} + {n-1\choose m}\kappa_\theta^{m},
\end{align*}
for $m<n$, and
\begin{align}
    S=\kappa_r\kappa_\theta^{n-1},
\end{align}
for $m=n$. Therefore, for $2\leq m<n$,
\begin{align}
\label{S-eq}
    S & = {n-1\choose m-1}\frac{1}{\chi W^3} \bigg(\dot r \frac{d}{d\tau} (\chi^2 \dot s) -\chi^2 \dot s \frac{d}{d\tau}\dot r +  \chi'\chi^3 \dot s^3 \bigg) \bigg( \frac{1}{W}\chi \dot s \frac{\xi'(r)}{\xi(r)}\bigg)^{m-1} \nonumber
    \\
    &+ {n-1\choose m}\bigg( \frac{1}{W}\chi \dot s \frac{\xi'(r)}{\xi(r)}\bigg)^{m} 
\end{align}
and for $m=n$,
\begin{align}
    S=\frac{1}{\chi W^3} \bigg(\dot r \frac{d}{d\tau} (\chi^2 \dot s) -\chi^2 \dot s \frac{d}{d\tau}\dot r +  \chi'\chi^3 \dot s^3 \bigg)\bigg(\frac{1}{W}\chi\dot s \frac{\xi^\prime(r)}{\xi(r)}\bigg)^{n-1}
\end{align}
Now, since
\begin{align}
\label{X.N-Formula}
    \langle X, N\rangle = \frac{\chi}{W}\dot r ,
\end{align}
the soliton equation for $\alpha\in\{1/m,1\}$, that is,
\[
S = c^{1/\alpha} \langle X, N\rangle^{1/\alpha},
\]
may be written as
\begin{align*}
  &   {n-1\choose m-1}\frac{1}{\chi W^3} \bigg(\dot r \frac{d}{d\tau} (\chi^2 \dot s) -\chi^2 \dot s \frac{d}{d\tau}\dot r +  \chi'\chi^3 \dot s^3 \bigg) \bigg( \frac{1}{W}\chi \dot s \frac{\xi'(r)}{\xi(r)}\bigg)^{m-1} \nonumber
    \\
    &+ {n-1\choose m}\bigg( \frac{1}{W}\chi \dot s \frac{\xi'(r)}{\xi(r)}\bigg)^{m} = \bigg( \frac{1}{W}c\chi\dot r \bigg)^{1/\alpha}
\end{align*}
for $2\leq m <n$, and
\begin{align*}
    \frac{1}{\chi W^3} \bigg(\dot r \frac{d}{d\tau} (\chi^2 \dot s) -\chi^2 \dot s \frac{d}{d\tau}\dot r +  \chi'\chi^3 \dot s^3 \bigg)\bigg(\frac{1}{W}\chi\dot s \frac{\xi^\prime(r)}{\xi(r)}\bigg)^{n-1}=\bigg( \frac{1}{W}c\chi\dot r \bigg)^{1/\alpha}
\end{align*}
for $m=n$.

If we assume in the previous calculations that $\tau$ is the arc-length parameter of the profile curve $\alpha$ of $M$ parametrized by
\[
\tau \mapsto (r(\tau), s(\tau)),
\]
then
$\dot r^2 + \chi^2 (r) \dot s^2 = 1,$
that is, $W\equiv 1$. Moreover, we can define a smooth determination of the angle $\phi$ between the tangent to the curve $\alpha$ at a given point and the coordinate direction $\partial_r$. One has
\[
\cos \phi = \langle \dot r \partial_r + \dot s \partial_s, \partial_r\rangle = \dot r. 
\]
In fact, one defines $\phi$ such that
\begin{align}
&     \dot r = \cos \phi , \\
& \chi(r) \dot s = \sin \phi.
\end{align}
Hence one has
\begin{align*}
    \dot r \frac{d}{d\tau} (\chi \dot s) -\chi \dot s \frac{d}{d\tau}\dot r = \cos^2\phi \dot \phi + \sin^2 \phi \dot \phi = \dot \phi.
\end{align*}
Therefore
\begin{align}
  &   \chi(r)\kappa_\tau = \dot r \frac{d}{d\tau} (\chi^2 \dot s) -\chi^2 \dot s \frac{d}{d\tau}\dot r +  \chi'\chi^3 \dot s^3 = \chi(r) \dot \phi + \chi(r) \chi'(r) \dot r^2 \dot s  +  \chi'(r)\chi^3(r) \dot s^3 \nonumber\\
    \nonumber
   &  = \chi(r)\dot \phi  + \chi(r)\chi'(r)\dot s(\dot r^2 + \chi^2\dot s^2) = \chi(r)\dot \phi  + \chi(r)\chi'(r)\dot s  
    =  \chi(r)\dot \phi   + \chi'(r)\sin \phi. \nonumber 
    \end{align}
Hence,
\begin{align}
    \kappa_\tau = \dot \phi + \frac{\chi'(r)}{\chi(r)}\sin \phi.
    \label{kappa-eq}
\end{align}
We also have
\begin{align}
    \label{kappa-eq-2}
    \kappa_\theta = \frac{\xi'(r)}{\xi(r)} \sin \phi.
\end{align}
Therefore, \eqref{S-eq} implies that
\begin{align*}
    & S= 
     {n-1\choose m-1}\bigg(\dot\phi   +\frac{\chi^\prime(r)}{\chi(r)}\sin\phi\bigg) \bigg(  \frac{\xi'(r)}{\xi(r)}\sin\phi\bigg)^{m-1} 
    + {n-1\choose m}\bigg(  \frac{\xi'(r)}{\xi(r)}\sin\phi\bigg)^{m} \\
   &  = {n-1\choose m-1} \left( \frac{\xi'(r)}{\xi(r)}\right)^{m-1}  \sin^{m-1}\phi\, \dot\phi + \sin^{m}\phi \bigg( {n-1\choose m-1}\frac{\chi'(r)}{\chi(r)} + {n-1\choose m} \frac{\xi'(r)}{\xi(r)}\bigg)\bigg( \frac{\xi'(r)}{\xi(r)}\bigg)^{m-1}\cdot
\end{align*}
For the special case $m=n$, we have
\begin{align*}
  &   S 
      = \bigg(\dot\phi + \frac{\chi'(r)}{\chi(r)}\sin\phi\bigg) \bigg(  \frac{\xi'(r)}{\xi(r)}\sin\phi\bigg)^{n-1} \\
& \,\,      = \left(\frac{\xi^\prime(r)}{\xi(r)}\right)^{n-1}\sin^{n-1}\phi \, \dot \phi + 
   \sin^n\phi \frac{\chi^\prime(r)}{\chi(r)}\left(\frac{\xi^\prime(r)}{\xi(r)}\right)^{n-1}\cdot
\end{align*}
If $M$ is the Killing cylinder $\{\Phi(s, x): {\rm dist}_{\mathbb{P}} (o, x) = r\}$ over the geodesic sphere of radius $r$ centered at $o$, we have $\phi\equiv \pi/2$. We conclude from the expression above  that the higher order mean curvature of this cylinder is the term between brackets, that is,  
\begin{align}
\label{Sm-formula-}
   &  S_m(r)=   \bigg[ {n-1\choose m-1}\frac{\chi'(r)}{\chi(r)} + {n-1\choose m} \frac{\xi'(r)}{\xi(r)}\bigg]\bigg( \frac{\xi'(r)}{\xi(r)}\bigg)^{m-1},\quad \mbox{ if } \quad 2\le m \le n-1,\\
\label{S-nformula-}
    & S_n(r)=\frac{\chi^\prime(r)}{\chi(r)}\left(\frac{\xi^\prime(r)}{\xi(r)}\right)^{n-1}\cdot
\end{align}
Rewriting the higher order mean of a rotationally invariant hypersurface $M$ in terms of the radial function $S_m(r)$, one gets
\begin{align}
\label{S-formula}
    & S= {n-1\choose m-1} \left( \frac{\xi'(r)}{\xi(r)}\right)^{m-1}  \sin^{m-1}\phi\, \dot\phi + S_m(r) \sin^{m}\phi ,\quad m\in \{2,\dots,n\}.
\end{align}
We conclude from \eqref{sol-eq-scalar}, \eqref{X.N-Formula} and \eqref{S-formula} that $M$ is a $m$-th mean curvature flow soliton  for $\alpha\in \{1/m, 1\}$ if its cylindrical coordinates satisfy the following first order ODE system:
\begin{align}
\label{Soliton-AngleAlpha}
    \begin{cases}
       & \dot r = \cos \phi\\
       & \chi(r) \dot s = \sin \phi\\
       & {n-1\choose m-1} \left( \frac{\xi'(r)}{\xi(r)}\right)^{m-1}  \sin^{m-1}\phi\, \dot\phi  = (c\chi\cos\phi)^{1/\alpha} - S_m(r) \sin^{m}\phi, \ \ m\in\{2,\dots, n\}.
     \end{cases}
\end{align}

\begin{remark}
    \label{remark-symmetries-reflection-orientation}
    Observe that if $\tau\mapsto(r(\tau), \phi(\tau),s(\tau))$ is a solution of the ODE system \eqref{Soliton-AngleAlpha} for $\tau\in(-\epsilon,\epsilon)$, then
\begin{itemize}
        \item for $m$ even we have that
        \begin{align}
        \label{reflection-symmetry}
            \tau\mapsto (r(\tau),-\phi(\tau),-s(\tau))
        \end{align}
        is also a solution. It corresponds to a reflection of the hypersurface with respect to $\mathbb{P}_0$ followed by a changing sign of orientation. In fact, this reflection does not change sign of $\kappa_\theta$, nor $\kappa_\tau$, while changing sign of orientation changes sign of both $\kappa_\theta$ and $\kappa_\tau$, which implies that $S$ does not change sign
        since $m$ is even. These two opperations change sign of $\langle X,N \rangle$ twice, preserving \eqref{sol-eq-scalar}. We call reflection the map \eqref{reflection-symmetry} or, equivalently, a reflection with respect to some $\mathbb{P}_{s_0}$ followed by a changing sign of orientation.
        \item For $m$ odd we have that
        \begin{align}
        \label{reversion-of-orientation}
            \tau\mapsto (r(-\tau),\phi(-\tau)+\pi,s(-\tau))
        \end{align}
        is also a solution. It corresponds to a changing sign of orientation of the hypersurface. In fact, changing sign of orientation changes sign of $\langle X,N\rangle$ and both $\kappa_\theta$ and $\kappa_\tau$, therefore changes sign of $S$, since $m$ is odd, preserving \eqref{sol-eq-scalar}. We call reversion of orientation the map \eqref{reversion-of-orientation} or, equivalently, a changing sign of orientation.
\end{itemize}
To avoid overwriting, we describe each soliton up to reflection or reversion of orientation in Theorems \ref{ExistenceTheorem-m<nXnonparallel} and \ref{ExistenceTheorem-m=nXparallel}.
\end{remark}

\begin{notation}
\label{notation-Sm}
    In order to unify $S_m$ expressions \eqref{Sm-formula-} and \eqref{S-nformula-}, we set, from now on,
    \begin{align*}
        {n-1 \choose n} =0.
    \end{align*}
    Therefore, the cylindrical $m$-th mean curvature is, for $m=2,...,n$,
    \begin{align}
    \label{Sm-formula-general}
        S_m(r)=   \bigg[ {n-1\choose m-1}\frac{\chi'(r)}{\chi(r)} + {n-1\choose m} \frac{\xi'(r)}{\xi(r)}\bigg]\bigg( \frac{\xi'(r)}{\xi(r)}\bigg)^{m-1}.
    \end{align}
\end{notation}

\begin{remark}
\label{Sm-behavior}
    The Structural Condition \ref{StructuralConditionxi2} implies that $S_m(r)$ is a  non-increasing function, since
    \begin{align*}
        &S_m^\prime = \left[{n-1 \choose m-1} \left(\frac{\xi^\prime\chi^\prime}{\xi\chi}\right)^\prime +2{n-1 \choose m} \frac{\xi^\prime}{\xi}\left(\frac{\xi^\prime}{\xi}\right)^\prime \right]\left(\frac{\xi^\prime}{\xi}\right)^{m-2}
        \\
        & \quad + (m-2)\left[{n-1 \choose m-1} \frac{\chi^\prime}{\chi} +{n-1 \choose m} \frac{\xi^\prime}{\xi}\ \right]\left(\frac{\xi^\prime}{\xi}\right)^{m-2} \left(\frac{\xi^\prime}{\xi}\right)^\prime \leq 0.
    \end{align*}
    It follows from the  Structural Condition \ref{StructuralConditionrto0} that as $r\to 0$, we have
    \begin{align*}
        &S_m(r) = {n-1 \choose m} r^{-m} + O(r^{-(m-1)}) \mbox{ for } m <n, \\
        &S_n(r) = r^{-(n-2)} +  O(r^{-(n-3)}) \mbox{ for } m=n \mbox{ and } X \mbox{ nonparallel}.
    \end{align*}
    The Structural Condition \ref{StructuralConditionrtoinfty} implies that as $r\to+\infty$, we have, in terms of the asymptotic behavior of the metric expressed in definition \ref{notation},
    \begin{itemize}
        \item for the Euclidean model and $m<n$,
        \begin{align*}
            S_m(r)={n-1 \choose m}r^{-m} + O(r^{-(m+1)});
        \end{align*}
        \item for the Riemannian product model and $m<n$,
        \begin{align*}
            S_m(r) = S_\infty + O(r^{-2}),
        \end{align*}
        where $S_\infty$ is a positive constant;
        \item for the hyperbolic model and $m\leq n$,
        \begin{align*}
            S_m(r) = S_\infty + O(e^{-r}),
        \end{align*}
        where $S_\infty$ is a positive constant.
    \end{itemize}
\end{remark}

\section{Main existence results: statement and proof}
\label{Section4}

In the following statement, we encompass existence results for higher-order mean curvature solitons for all different choices of $m$, $\alpha$ and $X$.

\begin{theorem}
\label{ExistenceTheorem-m<nXnonparallel}
  Suppose that the functions $\xi$ and $\chi$ in \eqref{riem-amb-met} satisfy the structural conditions \ref{StructuralConditionChi1} to \ref{StructuralConditionrtoinfty}.  Let $m\in\{2,\dots,n-1\}$ or $m=n$ and $X$ nonparallel. Let $\alpha\in\{1,1/m\}$. Then the  rotationally symmetric $m$-th mean curvature flow solitons  with respect to $X$ and $\alpha$ are classified as follows:
    \begin{itemize}
        \item[(a)] $m$-bowl solitons $\Sigma_0$, which are properly embedded, strictly convex and entire graphs contained in a half-space $\mathbb{P}\times[s_0,+\infty)$, with unbounded height and whose orientation $N$ is such that $\langle X,N\rangle$ is positive. Moreover, $\Sigma_0$ is $C^2$ regular, and the profile curve $s(r)$ satisfies, as $r\to 0$,
    \begin{align}
    \label{bowl-expansion}
        s(r) = s_0 + \frac{1}{2} c^{\frac{1}{m\alpha}} {n\choose m}^{-1/m} \, r^2 + O(r^3).
    \end{align}
    Here, $s_0$ is the minimum height of $\Sigma_0$.
    \item[(b)]
            For $\alpha=1/m$ and $m$ even, there are two one-parameter families $\mathcal{C}^i=\{\Sigma^i_{r_0}; \ r_0>0\}$, $i\in\{1,2\}$, of properly embedded annular $m$-th mean curvature flow solitons with nonempty boundary, with the following properties:
            \begin{itemize}
            \item[1.] For $i=1,2$ and for each $r_0>0$, the soliton $\Sigma^i_{r_0}$ is the graph of an unbounded  function defined in  $\mathbb{P} \backslash B(o,r_0)$ contained in a 
            half-space $\mathbb{P}\times[s_0,+\infty)$, where $s_0$ is the minimum height of $\Sigma_{r_0}^i$.
            \item[2.] Along its boundary, $\Sigma^1_{r_0}$ {\rm(}respectively, $\Sigma^2_{r_0}${\rm)}  is $C^1$ {\rm(}respectively, $C^2${\rm)} and tangent {\rm(}respectively, perpendicular{\rm)} to the horizontal hyperplane $\mathbb{P}_{s_0}$; moreover, the orientation of $\Sigma^1_{r_0}$ {\rm(}respectively, $\Sigma^2_{r_0}${\rm)} given by the normal vector field $N$ satisfies $\langle X,N\rangle >0$ everywhere {\rm(}respectively, $\langle X, N\rangle\ge 0$ everywhere, with equality holding only at $s=s_0${\rm)}.  
            We refer to $\Sigma_{r_0}^2$, $r_0>0$, as translating catenoids.
            \item[3.] For $m=n$, there exists an extra family of conic, proper, embedded and rotationally symmetric solutions $\mathcal{C}^1_n=\{\Sigma^1_{\phi_0}; \, \phi_0 \in (0, \pi/2)\}$, where $\Sigma^1_{\phi_0}$ is the graph of a rotationally invariant unbounded function $s(r)$ in $\mathbb{P}$  satisfying $\tan\phi_0 = s^\prime(0)$. Moreover, $\Sigma^1_{\phi_0}$  is contained in a half space $\mathbb{P}\times[s_0,+\infty)$.
    \end{itemize}
    \item[(c)]
        For $\alpha\in\{1/m,1\}$ and
        $m$ odd, there exists a one-parameter family $\mathcal{C}^3=\{\Sigma^3_{r_1}; \ r_1>0\}$ of properly embedded annular rotational $m$-th mean curvature flow solitons, with the following properties:
        \begin{itemize}
            \item[1.] Each soliton
           $\Sigma_{r_1}^3$  is the the union of two graphs $\Sigma_{r_1}^+$ and $\Sigma_{r_1}^-$ of unbounded functions defined in  $\mathbb{P}\backslash B(o,r_1)$ satisfying $\partial \Sigma^{\pm}_{r_1} = \partial B(o,r_1)\times\{s_1\}\subset \mathbb{P}_{s_1}$ for some $s_1 \in \mathbb{R}$; moreover, $\Sigma_{r_1}^3$ is perpendicular to  $ \mathbb{P}_{s_1}$ along $\partial \Sigma^{\pm}_{r_1}$.
            \item[2.] 
           $\Sigma_{r_1}^3$ is  contained in the half space $\mathbb{P}\times[s_0,+\infty)$, where $s_0$ is the minimum height of $\Sigma^3_{r_1}$ and $s_1-s_0$ depends on $r_1$; moreover,  $\Sigma_{r_1}^3\cap\mathbb{P}_{s_0}=\partial B(o,r_0)\times \{s_0\}$ for some $r_0>r_1$ that also depends on $r_1$. These solitons are $C^2$ hypersurfaces except along this intersection with $\mathbb{P}_{s_0}$.
            \item[3.]
            If $m<n$, for every $r_0>0$ there exists $\Sigma^3_{r_1}$, with $r_1\in(0,r_0)$, as described above. If $m=n$, the same fact holds if and only if  $r_0> r_n$ for some $r_n>0$ that depends on $n$. 
            \item[4.]
            For $m=n$, there exists extra two one-parameter families of conic, proper, embedded and rotationally symmetric solutions, $\mathcal{C}^1_n$ as in (a) and $\mathcal{C}^3_n = \{\Sigma^3_{\phi_0}; \, \phi_0\in (\pi/2, \pi)\}$, where $\Sigma_{\phi_0}^3$ is the graph of a rotationally symmetric function $s(r)$ in $\mathbb{P}$ satisfying $s^\prime(0)=\tan \phi_0$; moreover, there exists a decreasing function $r_0(\phi_0)$ with $0 < r_0(\phi_0)\le r_n$ such that $s(r)$ is increasing and unbounded for $r>r_0(\phi_0)$, is bounded for $r \in [0, r_0(\phi_0)]$ whenever $\phi\not \equiv \pi/2$ (mod $\pi$) and satisfies  $s^\prime(r_0(\phi_0))=0$. 
    \end{itemize}
    \item[(d)]
            For $\alpha=1$ and $m$ even, there exists two one-parameter families of $m$-th mean curvature solitons: a family $\mathcal{C}^1$ whose description is exactly as in (b) above; and the family $\mathcal{C}^4=\{\Sigma^4_{r_1}; \ r_1>0\}$ of properly embedded annular rotationally symmetric $m$-th curvature flow solitons, with the following properties: 
                \begin{itemize}
                \item[1.] For each $r_1>0$, the soliton $\Sigma_{r_1}^4$ is contained in the complement of the cylinder
                 $ B(o,r_1)\times\mathbb{R}$ and on the half space $\mathbb{P}\times[s_0,+\infty)$, where $s_0$ is the minimum height of $\Sigma_{r_1}^4$.
                \item[2.] $\Sigma_{r_1}^4$ is tangent to $\mathbb{P}_{s_0}$ and it is $C^2$-singular along the intersection $\partial B(o, r_0) \times \{s_0\} = \Sigma^4_{r_1}\cap \mathbb{P}_{s_0}$, where $r_0$ depends on $r_1$; $\Sigma_{r_1}^4$ is  orthogonal to $\mathbb{P}_{s_1}$ along the intersection $\partial B(o, r_1) \times \{s_1\} = \Sigma^4_{r_1}\cap \mathbb{P}_{s_1}$, for some $s_1 \in\mathbb{R}$ such that  $s_1 - s_0$ also depends on $r_1$. The soliton $\Sigma_{r_1}^4$ is a $C^2$ hypersurface except at its points of minimum height.
                \item[3.] The orientation $N$ of $\Sigma_{r_1}^4$ is such that $\langle \partial_r,N\rangle \leq 0$, the equality being valid only at the points of $\Sigma^4_{r_1}\cap \mathbb{P}_{s_0}$. 
                \item[4.]
                If $m<n$, for every $r_0>0$ there exists $\Sigma^4_{r_1}\in \mathcal{C}^4$, with $r_1\in(0,r_0)$, as in the configuration described above. If $m=n$, the same fact holds if and only if $r_0 > \rho_n$ for some $\rho_n > 0$ that depends on n.
                \item[5.]
                For $m=n$, there exists extra two one-parameter families of conic, proper, embedded and rotationally symmetric solutions: $\mathcal{C}^1_n$ as in (b); and $\mathcal{C}^4_n = \{\Sigma^4_{\phi_0}; \, \phi_0\in [\pi/2, \pi)\}$, where $\Sigma_{\phi_0}^4$ is the graph of a rotationally symmetric function $s(r)$ in $\mathbb{P}$  satisfying $s^\prime(0)=\tan \phi_0$ and $s^\prime(r_0(\phi_0))=0$; here,   $r_0(\phi_0)$ is a decreasing function of $\phi_0$ with $0 < r_0(\phi_0)\le \rho_n$. The function $s(r)$ is bounded for $\phi\not \equiv \pi/2$ (mod $\pi$) and unbounded otherwise.
    \end{itemize}
    \end{itemize}
\end{theorem}

In the following sections, we  must consider existence results for solitons in the cases $\alpha\in\{1/m,1\}$ for $m$ odd and $m$ even separately. In any case, the first order ODE system \eqref{Soliton-AngleAlpha} may be written as
\begin{align}
\label{ODE-Matrix}
    \begin{pmatrix}
    \dot r \\
    \dot \phi
    \end{pmatrix}
    =
    \begin{pmatrix}
    \cos \phi\\
    {n-1 \choose m-1}^{-1} \big(\frac{\xi}{\xi'}\big)^{m-1} (c\chi)^{1/\alpha} \sin\phi\Big(\frac{\cos^{1/\alpha}\phi}{\sin^{m}\phi} -\frac{S_m(r)}{(c\chi)^{1/\alpha}}  \Big)
    \end{pmatrix}
    =: M(r,\phi),
\end{align}
an autonomous system. For each pair $(m,\alpha)$ we define a phase portrait $\Theta_{m,\alpha}$, an open set of  $\{(r,\phi)\in\mathbb{R}_+\times\mathbb{R}; \ \phi \neq k\pi, \, k\in \mathbb{Z}\}$ where the system is $C^2$. Note that $\phi = k\pi$, $k\in \mathbb{Z}$, correspond to the totally geodesic hypersurfaces $\mathbb{P}_s$, $s\in \mathbb{R}$, which are not solitons.

By definition, at an equilibrium point $(r_1,\phi_1)\in \Theta_{m, \alpha}$ one has $M(r_1,\phi_1)=0$, which implies $S_m(r_1) = 0$. In the cases $m<n$ or $m=n$ and nonparallel $X$, no equilibrium point exists and, in particular, there is no translating soliton cylinder.

A straightforward consequence of uniqueness of the Cauchy problem for the first order ODE system \eqref{ODE-Matrix} is that the trace of its orbits  are regular proper $C^2$ curves of $\Theta_{m, \alpha}$. Therefore, for any $(r_0,\phi_0)\in \Theta_{m, \alpha}$, there is a unique $C^2$ orbit $\tau\mapsto\gamma(\tau)=(r(\tau),\phi(\tau))$ such that $\gamma(0)=(r_0,\phi_0)$. Moreover,  the properness condition means  that any orbit $\gamma(\tau) = (r(\tau),\phi(\tau))$ cannot have as endpoints  points $(r_1,\phi_1)$ in the interior of $\Theta_{m, \alpha}$, since at these points one has existence and uniqueness of solutions of \eqref{ODE-Matrix}, and thus any orbit around $(r_1,\phi_1)$ could be extended. Hence,  this properness condition implies that any orbit $\gamma(t)$ is a maximal curve contained in $\Theta_{m, \alpha}$ which has its endpoints at the boundary $\partial \Theta_{m, \alpha}$ and splits $\Theta_{m, \alpha}$ into two disjoint open sets. Therefore, another orbit distinct of $\gamma$  that contains a point in one of those open sets is  entirely contained in this open set, since two distinct orbits do not cross. It follows from \eqref{Soliton-AngleAlpha}  that whenever either $\dot r \neq 0$ or $\dot \phi \neq 0$, these orbits may be written as graphs of functions $r = r(\phi)$ or $\phi = \phi(r)$, respectively,  satisfying 
\begin{align}
\label{F-Big-ODE}
   \frac{dr}{d\phi} =  
     {n-1 \choose m-1} \left(\frac{\xi^\prime(r)}{\xi(r)}\right)^{m-1} (c\chi(r))^{-1/\alpha}\cot \phi\left(\frac{\cos^{1/\alpha}\phi}{\sin^m\phi} - \frac{S_m(r)}{(c\chi(r))^{1/\alpha}}\right)^{-1}
    =: F(r,\phi)
    \end{align}
    or
    \begin{align}
    \label{1/F-Big-ODE}
     \frac{d\phi}{dr} = \frac{1}{F(r,\phi)} = {n-1 \choose m-1}^{-1} \left(\frac{\xi}{\xi^\prime}\right)^{m-1} (c\chi)^{1/\alpha} \tan\phi \left(\frac{\cos^{1/\alpha}\phi}{\sin^m\phi} - \frac{S_m}{(c\chi)^{1/\alpha}} \right) . 
\end{align}
From now on, we denote the regions $(0,\pi/2]$, $[\pi/2,\pi)$, $(-\pi,-\pi/2]$ and $[-\pi/2,0)$ of $\Theta_{m, \alpha}$ by $Q_1$, $Q_2$, $Q_3$ and $Q_4$, respectively. We refer to these regions as quadrants of the phase portrait.

\begin{remark}
    \label{rmk-phasespace}
    For $m$ even and $\alpha=1/m$, observe that $c \chi \cos \phi =c\langle X,N\rangle=S^{1/m}\geq0$, therefore we may consider $\phi\in Q_1\cup Q_4$ and the phase space is given by $\Theta_{m,\alpha}=(0,+\infty)\times (Q_1 \cup Q_4)$. For $m$ even and $\alpha=1$ or $m$ odd, $\cos\phi$ and $S$ have same sign and we consider $\Theta_{m,\alpha}=(0,+\infty)\times (Q_1 \cup Q_2 \cup Q_3\cup Q_4)$.
\end{remark}

In order to proceed with the analysis of the phase portrait and behavior of the orbits, we define the following auxiliary function $f:[0,1)\to[0,+\infty)$
\begin{align}
\label{f-function}
    f(x) = \frac{x^{\frac{1}{\alpha}}}{(1-x^2)^{\frac{m}{2}}}\cdot 
\end{align}
We observe that the definition of $f$ depends on $m$ and $\alpha$, but we prefer to not use indices in its notation, since these parameters are clear from the context. The following lemma will be useful in the proof of the subsequent existence results.

\begin{lemma}
\label{lemma-fdiffeomorphism}
    $f:(0,1)\to(0,+\infty)$ is a strictly increasing diffeomorphism, convex for $\alpha=1/m$ and $\underset{x\to0}{\lim}\, f'(x)=1$ for $\alpha=1$.
\end{lemma}
\noindent\textit{Proof.} A straightforward computation shows that, for $0<x<1$, $\underset{x\to0}{\lim} \, f(x) = 0$, $\underset{x\to 1}{\lim} \, f(x) = +\infty$ and
\begin{align*}
f'(x)=\frac{1+x^2(m\alpha-1)}{\alpha x(1-x^2)}f(x)=\frac{1+x^2(m\alpha-1)}{\alpha (1-x^2)^{\frac{m+2}{2}}}x^{\frac{1}{\alpha}-1}>0.
    \end{align*}
    Hence, 
    \begin{align*}
     f''(x)& =\big\{ \big(1+x^2(m\alpha-1)\big)^2 \\
    & \quad-\alpha\big(2x^2(1-x^2)(m\alpha-1)-(1-3x^2)(1+x^2(m\alpha-1))\big) \big\}\frac{f(x)}{(\alpha x(1-x^2))^2} 
    \end{align*}
    what implies that
    \begin{align*}
     f''(x)=\frac{m+1-3x^2}{m}\frac{m^2f(x)}{( x(1-x^2))^2}>0 \mbox{ for } \alpha=1/m.
\end{align*}
This finishes the proof. 
\hfill $\square$

\begin{definition}
\label{Phima}
    We define the auxiliary function
\begin{align}
\label{f-Phima}
    \Phi_{m,\alpha}(r) = \arccos \left[f^{-1}\left(\frac{S_m(r)}{(c\chi(r))^{1/\alpha}}\right)\right],
\end{align}
which is a non-decreasing functions of $r$ {\rm(}see Remark \ref{Sm-behavior}{\rm)}. Therefore, we may use the expression
\begin{align}
    \frac{S_m(r)}{(c\chi(r))^{1/\alpha}} = f(\cos \Phi_{m,\alpha})
\end{align}
whenever it is convenient. 
\end{definition}

\begin{notation}
    Let $\gamma$ be an orbit passing through a point $(r_0,\phi_0)$, where $\phi_0$ is in the interior of some quadrant $Q_i$ and $r_0>0$. We call right component of $\gamma\setminus\{(r_0,\phi_0)\}$, and denote it by $\gamma^+$, the component that, for some $\epsilon>0$, $\gamma^+ \cap B_\epsilon((r_0,\phi_0))\subset (r_0,+\infty)\times Q_i$. The other component of $\gamma\setminus\{(r_0,\phi_0)\}$ we call left component and denote it by $\gamma^-$. This is well defined, since $\gamma$ is locally the graph of a function $\varphi(r)$ that satisfies \eqref{1/F-Big-ODE}.
\end{notation}

In order to prove Theorem \ref{ExistenceTheorem-m<nXnonparallel}, we set a list of propositions about orbits in the portrait space.

\begin{proposition}
\label{proposition-orbitTrapped}
    Let $m<n$ or $m=n$ and $X$ nonparallel. Let $\phi_0$ be a point in the interior of a quadrant $Q_i$, $r_0>0$ and $\gamma$ the orbit passing through $(r_0,\phi_0)$. If $\cos^{1/\alpha}\phi$ and $\sin^m\phi$ have same sign in $Q_i$, then the right component of $\gamma\setminus(r_0,\phi_0)$ is contained in $[r_0, +\infty)\times Q_i$ and this component is the graph of a function $\varphi:[r_0,+\infty)\to Q_i$ that satisfies \eqref{1/F-Big-ODE}.
\end{proposition}

\noindent \textit{Proof.} Let $\varphi(r)$ be a function defined in a neighborhood of $r_0$ whose graph is contained in $\gamma$. We rewrite \eqref{1/F-Big-ODE}, in this context, as
\begin{align}
\label{ODE-ffunction(+-)}
    \frac{d\varphi}{dr} = {n-1 \choose m-1}^{-1} \left(\frac{\xi}{\xi'}\right)^{m-1} (c\chi)^{1/\alpha} \tan\phi \left[f(|\cos\phi|) - f(\cos\Phi_{m,\alpha})\right].
\end{align}

If $m$ is even and $\alpha=1/m$, $\phi_0$ is in $Q_1$ or $Q_4$ (see Remark \ref{rmk-phasespace}), therefore $d\varphi/dr$ is positive if and only if $0<\varphi(r)<\Phi_{m,\alpha}(r)$, if $\phi_0\in Q_1$; or $-\pi/2\leq \varphi(r) <-\Phi_{m,\alpha}(r)$, if $\phi_0\in Q_4$. Similarly,  $d\varphi/dr$ is negative if and only if $\Phi_{m,\alpha}(r)<\varphi(r)\leq \pi/2$, if $\phi_0\in Q_1$; or $0< \varphi(r) <-\Phi_{m,\alpha}(r)$, if $\phi_0\in Q_4$. See Figure \ref{fig:alpha-m-even}. Suppose that at some $r=r_1>r_0$, the graph of $\varphi$ hits the boundary of $Q_i$. However this implies that, for $r<r_1$ sufficiently close to $r_1$, the sign of $d\varphi/dr$ contradicts what we have just established for it. Therefore $\varphi(r)$ can be extended for all $r\geq r_0$. 

If $m$ even and $\alpha=1$, the quadrants where $\cos \phi_0>0$ are, again, $Q_1$ and $Q_4$ and we proceed in the same way.

If $m$ is odd, the quadrants where $\cos\phi_0>0$ are $Q_1$ and $Q_3$. Therefore, $d\varphi/dr$ is positive if and only if $0<\varphi(r) < \Phi_{m,\alpha}(r)$, if $\phi_0\in Q_1$; or $-\pi<\varphi(r) < -\pi + \Phi_{m,\alpha}(r)$, if $\phi_0\in Q_3$. Similarly, $d\varphi/dr$ is negative if and only if $\Phi_{m,\alpha}(r) < \varphi(r) \leq \pi/2$, if $\phi_0\in Q_1$; or $-\pi + \Phi_{m,\alpha}(r) < \varphi(r) \leq -\pi/2$, if $\phi_0\in Q_3$. See Figure \ref{fig:enter-label}. Suppose that at some $r=r_1>r_0$, the graph of $\varphi$ hits the boundary of $Q_i$. However this implies that, for $r<r_1$ sufficiently close to $r_1$, the sign of $d\varphi/dr$ contradicts what we have just established for it. Therefore $\varphi(r)$ can be extended for all $r\geq r_0$. \hfill $\square$

\begin{proposition}
\label{proposition-orbitScapingQ}
    Let $m<n$ or $m=n$ and $X$ nonparallel. Let $\phi_0$ be a point in the interior of a quadrant $Q_i$, $r_0>0$ and $\gamma$ the orbit passing through $(r_0,\phi_0)$. If $\cos^{1/\alpha}\phi$ and $\sin^m\phi$ have opposite signs in $Q_i$, then the right component of $\gamma\setminus \{(r_0,\phi_0)\}$ is not contained in this set. Therefore, there exists $r_1>r_0$ such that $\gamma^+$ intersects the boundary of $Q_i$ at $r=r_1$.
\end{proposition}

\noindent \textit{Proof.} Suppose that $\gamma^+$ is contained in $[r_0,+\infty)\times Q_i$. Then $\gamma^+$ is the graph  of a function $\varphi:(r_0,+\infty) \to Q_i$, since, in this context,
\begin{align*}
    &\frac{d\varphi}{dr} = {n-1 \choose m-1} \tan\varphi 
    \\
    &\qquad \times\left\{
    -\left(\frac{\xi}{\xi^\prime}\right)^{m-1} (c\chi)^{1/\alpha} f(|\cos \varphi|) -
    \left[
    {n-1 \choose m-1}\frac{d}{dr}\log \chi + {n-1 \choose m}\frac{d}{dr}\log\xi
    \right]
    \right\}
\end{align*}
has constant sign (see Notation \ref{notation-Sm}). If either $m<n$ or $m=n$ and the metric asymptotes the hyperbolic model, we have
\begin{align*}
    \frac{\pi}{2} &\geq |\varphi(R) - \varphi(r_0)| = \int_{r_0}^{R}\left|\frac{d\varphi}{dr}\right| \, dr  \\ 
    &> 
    {n-1 \choose m-1}
    \left[
    {n-1 \choose m-1}\log\frac{\chi(R)}{\chi(r_0)} + {n-1 \choose m}\log\frac{\xi(R)}{\xi(r_0)}
    \right] \underset{R\to+\infty}{\longrightarrow} +\infty,
\end{align*}
a contradiction. If $n=m$, $\alpha=1$ and $X$ parallel or nonparallel, we have
\begin{align*}
    \left| \frac{d\varphi}{dr} \right| \geq \left(\frac{\xi}{\xi^\prime}\right)^{n-1} (c\chi) \frac{1}{|\sin^{n-1}\varphi|} 
    \end{align*}
    what implies that
    \begin{align*}
    \left| \int_{\varphi_0}^{\varphi(R)} \sin^{n-1}\phi \, d\phi\right| \geq \int_{r_0}^{R} \left(\frac{\xi}{\xi^\prime}\right)^{n-1} (c\chi) \, dr \underset{R\to+\infty}{\longrightarrow} +\infty,
\end{align*}
a contradiction. If $m=n$, $\alpha=1/n$ and $n$ is even, $\cos^{1/\alpha}\phi$ and $\sin^m \phi$ are positive, therefore have the same sign (see Remark \ref{rmk-phasespace}). If $m=n$, $\alpha=1/n$ and $n$ is odd, $\cos^{n}\phi$ and $\sin^n \phi$ have opposite signs in $Q_2$ and $Q_4$. In both cases, $d\varphi/dr > 0$. By monotonicity, $\varphi$ asymptotes $\phi_1$ for some $\phi_0<\phi_1\leq \pi$, if $\varphi\in Q_2$, or $\phi_0<\phi_1\leq 0$, if $\varphi \in Q_4$. In both cases, there exists $C>0$ such that $|\cot^{n-1} \varphi| >C$, therefore
\begin{align*}
    \phi_1-\phi_0\geq\varphi(R)-\varphi(r_0)=\int_{r_0}^{R}\frac{d\varphi}{dr} \, dr > C \int_{r_0}^{R} \left(\frac{\xi}{\xi^\prime}\right)^{n-1} (c\chi)^n \, dr \underset{R\to+\infty}{\longrightarrow} + \infty,
\end{align*}
a contradiction.

\hfill $\square$

\begin{proposition}\label{proposition-endpoint-raxis}
    Given $r_0>0$, there exist unique orbits in $\mathbb{R}_+\times Q_1$, $\mathbb{R}_+\times Q_2$, $\mathbb{R}_+\times Q_3$ and $\mathbb{R}_+\times Q_4$ whose one of its endpoints is $(r_0,0)$, $(r_0,\pi)$, $(r_0,-\pi)$ and $(r_0,0)$, respectively.
\end{proposition}

\noindent \textit{Proof.} We present the proof of the proposition only in the case $\mathbb{R}_+\times Q_1$ since it can be easily adapted to the other cases. Given $r_1>r_0$, define $\delta_1=r_1-r_0$. Since $\lim_{\phi\to0}\dot\phi=+\infty$ is uniform in $r_0 \leq r \leq r_1$ by \eqref{F-Big-ODE}, fixing $R>0$, there exists $\tilde \epsilon>0$ such that $\dot\phi>R$ for all $(r,\phi)\in [r_0, r_1]\times(0,\tilde\epsilon]$.
Define $\epsilon_1=\min\{\tilde\epsilon, \delta_1 R\}$ and let $\gamma_1$ be the orbit that contains $(r_1,\epsilon_1)$. Let $\rho(\phi)$ be the function corresponding to $\gamma_1$, then $\rho'=\dot \rho / \dot \phi = \cos\phi / \dot\phi >0$, by \eqref{ODE-Matrix}. Suppose that there exists $\phi_1 \in [0, \epsilon_1)$ such that $\rho(\phi_1)=r_0$. Then $\phi_1\leq\phi\leq \epsilon_1$ implies  that $ \rho(\phi)\in[r_0,r_1]$ and the following chain of inequalities hold
\begin{align*}
    \frac{d\rho}{d\phi} = \frac{\dot r}{\dot \phi} = \frac{\cos\phi}{\dot \phi}< \frac{1}{R} \Rightarrow
    r_1-\rho(\phi) <\frac{\epsilon_1-\phi}{R} = \frac{\epsilon_1}{R}-\frac{\phi}{R}\leq\delta_1 - \frac{\phi}{R}=r_1-r_0-\frac{\phi}{R},
\end{align*}
   from what follows that
    \begin{align*}
        r_0=\rho(\phi_1)>r_0 + \frac{\phi_1}{R},
    \end{align*}
        a contradiction. We define $\tilde r_2 \in (r_0, r_1)$ as $\tilde r_2 = \lim_{\phi\to0} \rho(\phi)$, which exists by monotonicity. Then $(\tilde r_2,0)$ is one of the endpoints of $\gamma_1$. Now we proceed inductively, defining

\[
    r_k=\min\{\tilde r_k, r_0+2^{-k}\}, \ \ \ \delta_k = r_k - r_0, \ \ \ \mbox{and } \ \ \epsilon_k = \min\{\tilde\epsilon. \delta_k R\}.
\]
Therefore $0<\delta_k \leq 2^{-k}$. We define $\gamma_k$ as the orbit passing through $(r_k,\epsilon_k)$, that converges to $(\tilde r_{k+1}, 0)$ and we have
\[
    r_0 \leq \tilde r_{k+1} < r_k.
\]
Therefore $\tilde r_k \searrow r_0$. If $\rho_i:[0,\epsilon_1]\to [r_0,r_1]$ is the function whose graph is contained in the trace of $\gamma_i$, the sequence $\rho_1(\phi), \rho_2(\phi), \dots$ is monotone, since the orbits do not cross each other, and the functions are $C^2$ in $(0,\epsilon_1)$. Therefore, $\rho_i(\phi)$ converges pointwisely to a function $\rho(\phi)$. By \eqref{F-Big-ODE}, the convergence is uniform in each compact of $(0,\epsilon_1)$, therefore $\rho(\phi)$ is continuous. By \eqref{F-Big-ODE},
the derivatives $\rho_1^\prime(\phi), \rho_2^\prime(\phi), \dots$ converge uniformly in each compact of $(0,\epsilon_1)$, from which we conclude that $\rho(\phi)$ is differentiable and satisfies \eqref{F-Big-ODE}. Finally, the graph of $\rho(\phi)$ is contained in the trace of the desired orbit.

To show uniqueness, we compute
\begin{align*}
   &  \frac{\partial F}{\partial r} = -{n -1 \choose m-1} \cot \phi \\
    & \qquad \times \frac{(c\chi)^{1/\alpha}\big(\frac{\xi} {\xi^\prime}\big)^{m-2}\frac{\cos^{1/\alpha}\phi}{\sin^m\phi}\Big[\frac{1}{\alpha}\frac{\chi^\prime}{\chi} \frac{\xi}{\xi^\prime} + (m-1)\big(\frac{\xi}{\xi^\prime}\big)^\prime \Big] -  \Big[{n-1\choose m-1} \big(\frac{\chi^\prime}{\chi}\big)^\prime + {n-1 \choose m}\big(\frac{\xi^\prime}{\xi}\big)^\prime\Big]}{\left[ (c\chi)^{1/\alpha}\left(\frac{\xi}{\xi^\prime}\right)^{m-1}\frac{\cos^{1/\alpha}\phi}{\sin^m \phi} - \left[{n-1\choose m-1} \frac{\chi^\prime}{\chi} + {n-1 \choose m}\frac{\xi^\prime}{\xi}\right]\right]^2}\cdot 
\end{align*}
Since $\cos^{1/\alpha}\phi\, \sin^{-m} \phi$ is arbitrarily large, we can assume that $\frac{\partial F}{\partial r}<0$ for $(r,\phi) \in [r_0,r_1]\times(0, \epsilon_1]$.

let $r_-,r_+:[0,\epsilon_1)\to[r_0,+\infty]$ be smooth functions such that satisfy \eqref{F-Big-ODE} and $\lim_{\phi\to0}r_-(\phi)=r_0=\lim_{\phi\to0}r_+(\phi)$. Suppose that $r_{-}\neq r_{+}$. The uniqueness of orbits implies that $r_{-}(\phi)\neq r_{+}(\phi)$ for all $\phi>0$. Without loss of generality, assume that $r_{-}(\phi)<r_{+}(\phi)$ for all $\phi>0$. But this is a contradiction, since
\begin{align*}
    r_{-}^\prime(\phi)=F(r_-(\phi),\phi)>F(r_+(\phi),\phi)=r_{+}^\prime(\phi)   \mbox{ for all }   \phi>0.
\end{align*}
In the remaining cases, $\mathbb{R}_+\times Q_2$, $\mathbb{R}_+\times Q_3$ or $\mathbb{R}_+\times Q_4$, considering $(r_0,\pi)$, $(r_0,-\pi)$ or $(r_0,0)$ respectively, we set $r_1>r_0$ or $0<r_1<r_0$ if $\dot r / \dot\phi$ is positive or negative, respectively, and the proof is essentially the same. \hfill $\square$

\begin{remark}
\label{rmk-continuityofendpoints-orbits}
    Fixing $k \in \{-1,0,1\}$ and $0<r_1<r_2$, let $\gamma_1$ and $\gamma_2$ be orbits in the same connected component of $\Theta_{m,\alpha}$ which have one of their endpoints given by $(r_1, k\pi)$ and $(r_2,k\pi)$, respectively. If $r_1<r<r_2$, the orbit $\gamma$, in the same connected component, whose one endpoint is $(r,k\pi)$ is such that it is between $\gamma_1$ and $\gamma_2$, since two different orbits do not intersect. Conversely, an orbit $\gamma$ between $\gamma_1$ and $\gamma_2$ must have one of its endpoints $(r,k\pi)$ for some $r\in(r_1,r_2)$ for the same reason. We conclude that the endpoint changes continuously with the change of the orbit and vice versa. 
\end{remark}

\begin{proposition}
\label{propostion-noendpoint}
    For $m<n$, there is no orbit with an endpoint in $(r,\phi)=(0,\phi_0)$, where $\phi_0 \not\in \{-\pi,0,\pi\}$.
\end{proposition}

\noindent \textit{Proof}. Suppose that $\gamma$ is an orbit with one endpoint at $(0,\phi_0)$. Let us suppose that $\phi_0\in Q_1$. The proof is essentially the same for $\phi_0$ in $Q_2$, $Q_3$ or $Q_4$. For $\epsilon>0$ sufficiently small, the trace of $\gamma$ contains the graph of a decreasing function $\varphi:[0,\epsilon]\to Q_1$, where $\lim_{r\to0} \varphi(r) = \phi_0$. Let $\varphi(\epsilon)=\phi_1$. By \eqref{1/F-Big-ODE} and using the fact that $\varphi$ is a decreasing function,
one has
\begin{align*}
     \frac{d\varphi}{dr} =&  
     {n-1 \choose m-1}^{-1} \\
     &\times\left\{
     \frac{\cos^{\frac{1}{\alpha} - 1}\varphi}{\sin^{m-1}\varphi}\left(\frac{\xi(r)}{\xi^\prime(r)}\right)^{m-1} (c\chi(r))^{1/\alpha} -\tan\varphi
     \left[{n-1 \choose m-1} \frac{\chi^\prime}{\chi}+{n-1 \choose m} \frac{\xi^\prime}{\xi}\right]
     \right\}
     \\
     <&   
     {n-1 \choose m-1}^{-1} \bigg\{
     \frac{\cos^{\frac{1}{\alpha} - 1}\phi_1}{\sin^{m-1}\phi_1}\left(\frac{\xi(r)}{\xi^\prime(r)}\right)^{m-1} (c\chi(r))^{1/\alpha} \\
     &-\tan\phi_1
     \left[{n-1 \choose m-1} \frac{d}{dr} \log \chi+{n-1 \choose m} \frac{d}{dr} \log \xi\right]
     \bigg\} \Rightarrow \\
     \phi_1 - \phi_0 =& \int^{\phi_1}_{\phi_0} d\varphi 
     < 
     {n-1 \choose m-1}^{-1} \bigg\{
     \frac{\cos^{\frac{1}{\alpha} - 1}\phi_1}{\sin^{m-1}\phi_1}\int_{0}^{\epsilon}\left(\frac{\xi(\rho)}{\xi^\prime(\rho)}\right)^{m-1} (c\chi(\rho))^{1/\alpha} \,d\rho \\
     &-\tan\phi_1
     \left[{n-1 \choose m-1}  \log \chi+{n-1 \choose m}  \log \xi\right]_{0}^{\epsilon}
     \bigg\} = -\infty
\end{align*}
for $m<n$, which is a contradiction. \hfill $\square$

\begin{proposition}
\label{proposition-yesendpoint}
    Let $m=n$. For each quadrant $Q_i$ and each $\phi_0\in  Q_i$, there exists a unique orbit in $[0,+\infty)\times Q_i$ with an endpoint at $(0,\phi_0)$. Moreover, $(0,\pm\pi/2)$ is an equilibrium point.
\end{proposition}

\noindent \textit{Proof.} For $m=n$,
\begin{align*}
    S_n(r) = \frac{\chi^\prime}{\chi}\left(\frac{\xi^\prime}{\xi}\right)^{n-1} \Rightarrow
    \dot \phi =  \sin\phi \left[ \left(\frac{\xi}{\xi^\prime}\right)^{n-1} (c\chi)^{1/\alpha} \frac{\cos^{1/\alpha}\phi}{\sin^n\phi} - \frac{\chi^\prime}{\chi} \right]\to 0
\end{align*}
when $r\to0$. See Remark \ref{Sm-behavior} and Structural  Condition \ref{StructuralConditionrto0}. Therefore, the ODE system is well-defined in $\{0\}\times Q_i$. Then there exists a unique orbit that contains $(0,\phi_0)$ in $[0,+\infty)\times Q_i$. It is immediate to check that $(0,\pm\pi/2)$ is an equilibrium point.

\hfill $\square$

\begin{figure}[H]
    \centering
    \includegraphics[width=0.40\linewidth]{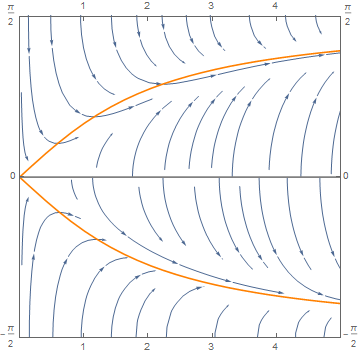}
    \includegraphics[width=0.40\linewidth]{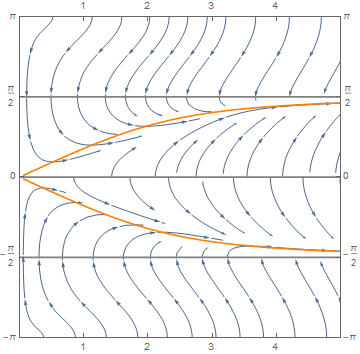}
    \caption{$\alpha=1/m$ (left), $\alpha=1$ (right) and $m$ even phase portraits. Blue oriented lines represent trace of orbits. Orange continuous lines are graphs of $\Phi_{m,\alpha}$ and $-\Phi_{m,\alpha}$.}
    \label{fig:alpha-m-even}
\end{figure}

\begin{figure}[H]
    \centering
    \includegraphics[width=0.30\linewidth]{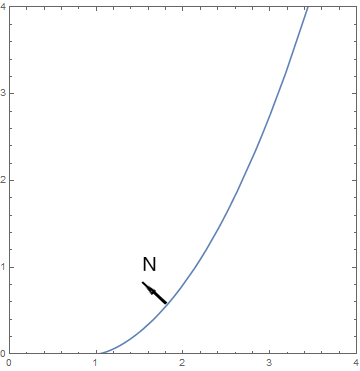}
    \includegraphics[width=0.30\linewidth]{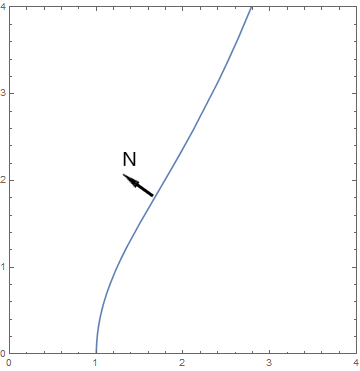}
    \includegraphics[width=0.232\linewidth]{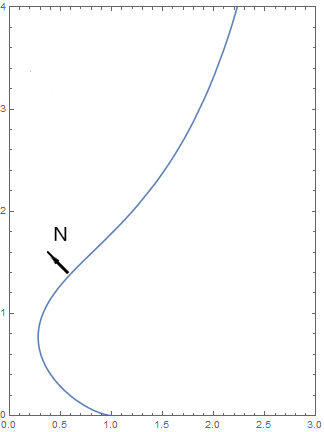}
    \caption{From left to right: profile curve of rotational translators in $\mathcal{C}^1$, $\mathcal{C}^2$ and $\mathcal{C}^4$, respectively.}
    \label{fig:profile-curves-m-even}
\end{figure}

\begin{figure}[H]
    \centering
    \includegraphics[width=0.40\linewidth]{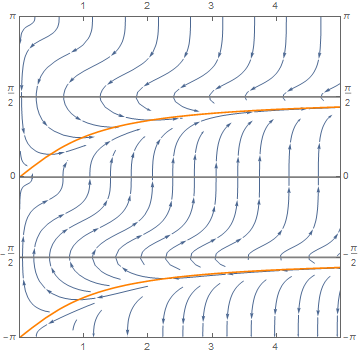}
     \includegraphics[width=0.225\linewidth]{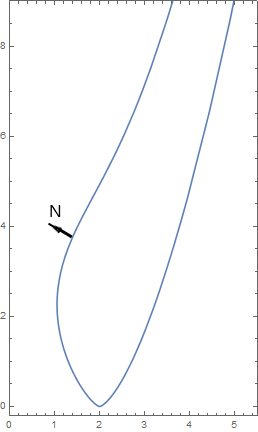}
    \caption{On left, $\alpha\in\{1/m,1\}$ and $m$ odd phase plane. Blue oriented lines represent trace of orbits. Orange continuous lines are graphs of $\Phi_{m,\alpha}$ and $\Phi_{m,\alpha}-\pi$. Profile curve of a rotational translator in $\mathcal{C}^3$ on right.}
    \label{fig:enter-label}
\end{figure}

\noindent \textit{Proof of Theorem \ref{ExistenceTheorem-m<nXnonparallel}, cases (b), (c) and (d)}. 
In order  to prove (b), we set $m$ even and $\alpha=1/m$ and then consider $\Theta_{m,\alpha}= (0,+\infty)\times(Q_1\cup Q_4)$. See Remark \ref{rmk-phasespace}. The translating soliton $\Sigma^1_{r_0}$ corresponds to an orbit in $(0,+\infty)\times Q_1$ or $(0,+\infty)\times Q_4$ with an endpoint at $(r_0,0)$, see Propositions \ref{proposition-orbitTrapped} and \ref{proposition-endpoint-raxis}. The translating soliton $\Sigma_{r_0}^{2}$ corresponds to the orbits starting at $(r_0,\pm \pi/2)$. In both cases, the orbit is the graph of a function $\varphi(r)$, $r\geq   r_0$, see Proposition \ref{proposition-orbitTrapped}. For each pair of orbits corresponding to $\Sigma^1_{r_0}$ or $\Sigma^2_{r_0}$, one of these two orbits is obtained from the other by reflection \eqref{reflection-symmetry}. See Remark \ref{remark-symmetries-reflection-orientation}.

To prove (c) we set $m$ odd. The translating soliton $\Sigma_{r_1}^3$ corresponds to the union of two distinct orbits with the same endpoint $(r_0,0)$, $r_0>r_1$ and $r_0$ depending on $r_1$, or the union of other pair of orbits, one with an endpoint in $(r_0,\pi)$ and the other with and endpoint in $(r_0,-\pi)$, $r_0>r_1$. These two situations correspond to a reversion of orientation \eqref{reversion-of-orientation}. We focus on the first situation. One of the orbits is $\gamma_-$  that contains $(r_1,-\pi/2)$, $r_1>0$. This orbit in $\mathbb{R}_+\times Q_4$ has an endpoint in $(r_0,0)$, see Proposition \ref{proposition-orbitScapingQ}, and in $\mathbb{R}_+\times Q_3$ it is defined for all $r>r_1$, see Proposition \ref{proposition-orbitTrapped}. The other orbit $\gamma_+$ is the one in $\mathbb{R}_+\times Q_1$ with an endpoint in $(r_0,0)$. By Proposition \ref{proposition-orbitTrapped}, $\gamma_+\subset \mathbb{R}_+\times Q_1$. These two orbits match orientation at $(r_0,0)$, therefore their union corresponds to a translating soliton. 

To prove the statement 3 of (c), and analogously the statement 4 of (d), we define, for $m$ odd and specifically in this context, $r_0=r_0(r_1)$, see Remark \ref{rmk-continuityofendpoints-orbits}. If $m<n$, suppose that there exists $r_n$ as in the statement. Then $\lim_{r_1\to 0} r_0(r_1) = r_n$. An orbit in $\mathbb{R}_+\times Q_4$ with an endpoint in $(r_0,0)$, $0<r_0<r_n$, must converge to some $(0,\beta(r_0))$, $\beta(r_0)\in (-\pi/2,0)$, since two orbits do not intersect in the phase space, which contradicts Proposition \ref{propostion-noendpoint}. If $m=n$ and $X$ is nonparallel, the orbit $\gamma_\beta$ starting at $(0,\beta)$, $\beta\in(-\pi/2,0)$, see Proposition \ref{proposition-yesendpoint}, have the other endpoint at $(r_0(\beta), 0)$ and $r_0(\beta)>0$ is decreasing, see Remark \ref{rmk-continuityofendpoints-orbits}. Therefore $r_0(\beta) \to r_n>0$ as $\beta \to -\pi/2$. Finally, the orbit $\gamma_n$ with an endpoint at $(r_n,0)$ in $\mathbb{R}_+\times Q_4$ converges to $(-\pi/2,0)$ and this orbit works as a barrier for orbits $\gamma_{r_0}$ in $\mathbb{R}_+\times Q_4$ with an endpoint at $(r_0,0)$: if $0<r_0<r_n$, $\gamma_{r_0}$ has another endpoint in $(0,\beta(r_0))$, $\beta(r_0)\in(-\pi/2,0)$, and if $r_0>r_n$, $\gamma_{r_0}$ pass through $(r_1(r_0),-\pi/2)$, $0<r_1(r_0)<r_0$, since two orbits do not intersect in the phase space. The proof of statement 4 in (c) is analogous to the proof of statement 5 of (d) which is presented in the sequel.

To prove (d) we set $m$ even and $\alpha=1$. The translating soliton $\Sigma_{r_1}^1$ corresponds to the orbit with an endpoint at $(r_1,0)$ in $\mathbb{R}_+\times Q_1$ or $\mathbb{R}_+\times Q_4$. See Propositions \ref{proposition-orbitTrapped} and \ref{proposition-endpoint-raxis}. The translating soliton $\Sigma_{r_1}^4$ corresponds to the orbit passing through $(r_1,\pm\pi/2)$. This orbit in $\mathbb{R}_+ \times Q_2$, respectively in $\mathbb{R}_+ \times Q_3$, is the graph of an increasing, respectively decreasing, function that hits a point $(r_0,\pm\pi)$ for some $r_0>r_1$. See Proposition \ref{proposition-orbitTrapped}. This orbit in $\mathbb{R}_+ \times Q_1$, respectively $\mathbb{R}_+ \times Q_4$, is defined for all $r>r_1$. See Proposition \ref{proposition-orbitScapingQ}. The proof of statement 4 in (d) is analogous to the proof of the statement 3 in (c). For each pair of orbits corresponding to a single solution in $\mathcal{C}^2$ or $\mathcal{C}^4$, one is obtained from the other by a reflection \eqref{reflection-symmetry}. Now we prove the statement 5 in (d). The hypersurfaces $\Sigma_{\phi_0}^1$ and $\Sigma_{\phi_0}^4$ correspond to the orbit starting at $(0,\phi_0)$, if $\phi_0\in Q_1 \cup Q_4$ or $\phi_0 \in Q_2 \cup Q_3$, respectively. Finally, $\Sigma_{\phi_0}^4$ is bounded for $\phi_0\neq  \pm \pi/2$  because $s'(r)= \tan\phi / \chi$ is bounded, and is unbounded for $\phi_0 =\pm \pi/2$, since
\begin{align*}
    |s(r) - s(r_-)| =& \left| \int_{r_-}^{r} \frac{ds}{dr} d\rho \right|
    = \left|\int_{r_-}^{r}  \tan\varphi \, \frac{d\varphi}{\varphi'} \right| 
    \\ 
    \geq & \left|
    \int_{\phi(r_-)}^{\phi(r)}
    {n-1 \choose m-1} \frac{1}{(c\chi)^{1/\alpha}}\left(\frac{\xi^\prime(\rho)}{\xi(\rho)}\right)^{n-1}  
    \frac{\sin^{m}\varphi}{ \cos^{1/\alpha} \varphi} \, d\varphi
    \right|
    \\
    \geq&
    \Lambda \left|
    \int_{\phi(r_-)}^{\phi(r)}  
     \frac{\sin^{m}\varphi}{ \cos^{1/\alpha} \varphi} \, d\varphi
    \right|
     \underset{\phi\to\pi/2}{\longrightarrow} +\infty,
\end{align*}
where
\begin{align*}
    \Lambda = \min_{r\in[r_0,r_1]} 
    \left\{ 
    {n-1 \choose m-1} \frac{1}{(c\chi)^{1/\alpha}}\left(\frac{\xi^\prime(\rho)}{\xi(\rho)}\right)^{n-1} 
    \right\}.
\end{align*}
These solutions are $C^2$-regular at points $(r,\phi)$, $\phi\not\in\{-\pi,0,\pi\}$, and $C^2$-singular otherwise by \eqref{kappa-eq}. The convexity is given by \eqref{kappa-eq}. This ends the proof. \hfill $\square$

\subsection{Bowl solitons for higher order mean curvature flows}

We start the proof of case (a) in Theorem \ref{ExistenceTheorem-m<nXnonparallel} with a preliminary existence result in the case when $X$ is a nonparallel vector field.

\begin{lemma}
\label{orbit-bowl-Xparallel}
    For $m=n$ and $X$ nonparallel, there exists an orbit in $\mathbb{R}_+\times Q_1$ with an endpoint at $(r,\phi)=(0,0)$.
\end{lemma}
\noindent \textit{Proof.} Let $\gamma_k$ be an orbit in $[0,+\infty)\times Q_1$ with an endpoint at $(r_k,0)$, where $r_k\searrow 0$, see Proposition \ref{proposition-endpoint-raxis}. Let $\rho_k:[0,\epsilon]\to[0,+\infty)$ be functions whose graphs correspond to $\gamma_k$. Therefore, $\rho_k$ are monotone sequences that converge pointwisely to $\rho:[0,\epsilon]\to [0,+\infty)$, with $\rho(0) = 0$. Since $\rho(\phi)$ satisfies \eqref{F-Big-ODE}, we conclude that it converges uniformly to $\rho(\phi)$ in compact intervals of $(0,\epsilon)$, see Proposition \ref{proposition-yesendpoint}. Therefore $\rho(\phi)$ satisfies \eqref{F-Big-ODE} and the corresponding orbit has an endpoint at $(0,0)$, which concludes the proof of the Lemma. \hfill $\square$

\medskip

Now, we follow with the the proof of case (a) that asserts the existence of bowl-solitons for higher order mean curvature flows.

\medskip

\noindent
\textit{Proof of Theorem \ref{ExistenceTheorem-m<nXnonparallel}, case (a)}. We fix the orbits in $\mathbb{R}_+\times Q_1$. To show uniqueness, we define
\[
    y= \cos\phi.
\]
If $(r, y(r))$ is the trace of an orbit, then
\begin{align}
\label{G-ODE-yr}
    \frac{dy}{dr} = {n-1 \choose m-1}^{-1} \left(\frac{\xi}{\xi^\prime} \right)^{m-1}(c\chi)^{1/\alpha} \frac{1-y^2}{y} \left[ f(\Gamma(r)) - f(y) \right] =: G(r,y),
\end{align}
where $\Gamma(r) = \cos \Phi_{m,\alpha}$ is defined implicitly by
\begin{align}
    \label{Gamma-function}
    \frac{\Gamma^{1/\alpha}(r)}{(1-\Gamma^2(r))^{m/2}} =f(\Gamma(r)) = \frac{S_m(r)}{(c\chi)^{1/\alpha}},
\end{align}
where $f$ is given by \eqref{f-function}. Therefore $\Gamma(r)$ is non-increasing, see Remark \ref{Sm-behavior} and Lemma \ref{lemma-fdiffeomorphism}, and by \eqref{G-ODE-yr},
\begin{align*}
    \frac{\partial G}{\partial y} =& 
    {n-1 \choose m-1}^{-1} \left(\frac{\xi}{\xi} \right)^{m-1}(c\chi)^{1/\alpha}  \left[ -\frac{1+y^2}{y^2}(f(\Gamma(r)) - f(y)) + \frac{1-y^2}{y}(-f^\prime(y)) \right] \\
    =&-{n-1 \choose m-1}^{-1} \left(\frac{\xi}{\xi} \right)^{m-1}(c\chi)^{1/\alpha}
    \\
    & \times \left[ \frac{1+y^2}{y^2}(f(\Gamma(r)) - f(y)) + \frac{1-y^2}{y} \frac{1+y^2(m\alpha-1)}{\alpha y (1-y^2)} f(y) \right] 
    \\
    =&-{n-1 \choose m-1}^{-1} \left(\frac{\xi}{\xi} \right)^{m-1}(c\chi)^{1/\alpha}
    \\
    &\times \left[ \frac{1+y^2}{y^2}f(\Gamma(r)) +  \left(\frac{1+y^2(m\alpha-1)}{\alpha y^2} -\frac{1+y^2}{y^2}\right) f(y) \right]  
    \\
    =&-{n-1 \choose m-1}^{-1} \left(\frac{\xi}{\xi} \right)^{m-1}(c\chi)^{1/\alpha}
    \\
    &\times \left[ \frac{1+y^2}{y^2}f(\Gamma(r)) +  \left(\frac{1}{\alpha} - (1+y^2) + y^2 \left(m-\frac{1}{\alpha}\right)\right) \frac{f(y)}{y^2} \right]  <0.
\end{align*}
Suppose that $y_-(r)$ and $y_+(r)$ are two distinct functions that satisfy \eqref{G-ODE-yr} and $\lim_{r\to0}y_{-}(r)= 0 =\lim_{r\to0}y_{+}(r)$. Therefore their graphs are trace of orbits in $(0,+\infty) \times Q_1$ with one end in $(0,0)$. Since two orbits do not intersect, we may assume that, without loss of generality, $y_-(r)<y_+(r)$ for $r>0$. However this leads to a contradiction, since
\[
    y^\prime_-(r)=G(r,y_-)>G(r,y_+)=y^\prime_+(r), \ \ \ r>0.
\]
Now we prove existence, which has been already proved for $m=n$ and $X$ nonparallel, see Lemma \ref{orbit-bowl-Xparallel}. We focus on the case $m<n$.
We observe that $y'(r)$ is positive, negative or null if and only if $\Gamma(r)>y(r)$, $\Gamma(r)<y(r)$ or $\Gamma(r) = y(r)$, respectively, by \eqref{G-ODE-yr}. Therefore, an orbit that contains a point in $\{(r,y); \, y>\Gamma(r)\}$, for $m<n$, must be contained in it. Otherwise, since $\Gamma(r)$ is non-increasing, when the graph of $y(r)$ intersects the graph of $\Gamma(r)$, $y(r)$ becomes non-decreasing, which is a contradiction.

Let $\gamma_{r_0}$ be the orbit in $(0,+\infty)\times Q_1$ with an endpoint at $(r_0,1)$, see Proposition \ref{proposition-endpoint-raxis}. Therefore, its trace may be written as a graph of $y_{r_0}:(r_0,+\infty)\to(0,1)$, where $\Gamma_\alpha(r)<y_{r_0}(r)<1$. We define the sequence of functions $\psi_{r_0}:(0,+\infty)\to(0,1)$
\begin{align*}
    \psi_{r_0}(r):=y_{r_0}(r_0+r).
\end{align*}
Observe that
\begin{align}
\label{psiprimer0}
    \psi^\prime_{r_0}(r) = y_{r_0}^\prime(r_0+r)=G(r_0+r,y_{r_0}(r_0+r))=G(r_0+r,\psi_{r_0}(r)).
\end{align}
Since $\chi'/\chi = O(r)$ and $\xi'/\xi = O(r^{-1})$ for $r$ sufficiently closed to $0$, one gets
\begin{align*}
    \frac{\partial G}{\partial r} =& {n-1 \choose m-1}^{-1} \frac{1-y^2}{y} \bigg\{ {n-1\choose m-1} \left(\frac{\chi'}{\chi}\right)' + {n-1\choose m} \left(\frac{\xi'}{\xi}\right)' \\
    & \qquad \qquad \qquad \quad - \left(\frac{\xi}{\xi'}\right)^{m-2}\chi^{\frac{1}{\alpha}}\left[(m-1)\left(\frac{\xi}{\xi'}\right)' + \frac{1}{\alpha}\left(\frac{\xi}{\xi'}\right)\left(\frac{\chi'}{\chi}\right)\right]
    \bigg\} \\
    =&{n-1 \choose m-1}^{-1} \frac{1-y^2}{y} \left[ -{n-1 \choose m}\frac{1}{r^2} + O(1) \right] \ \ \  \mbox{ as } r\to 0, \ \ \ \mbox{ since } m<n.
\end{align*}
Therefore $G(r,y)$ is a decreasing function of $y$ and, for some $r^*>0$, $r\mapsto G(r,y)$ is a decreasing function for $0<r<r^*$. For $r_0^*<r_0$ and $r+r_0<r^*$ we have $\psi_{r_0}(0)=1=\psi_{r_0^{*}}(0)$ and
\begin{align*}
    \psi^\prime_{r_0^{*}}(r) = G(r_0^{*}+r,\psi_{r_0^*}(r)) > G(r_0+r,\psi_{r_0^*}(r)) \ \ \ \mbox{and} \ \ \   
    \psi^\prime_{r_0}(r) =  G(r_0+r,\psi_{r_0}(r))
\end{align*}
for all $r<r^*-r_0$. By standard comparison argument, $\psi_{r_0^*}(r)\leq\psi_{r_0}(r)$ for all $r<r^*-r_0$. Moreover, $r\mapsto\psi_{r_0}(r)$ is a decreasing function and $0<\Gamma(r_0+r)<\psi_{r_0}(r)<1$. Hence, by \eqref{G-ODE-yr}, $\psi_{r_0}$ converges uniformly in any compact $I\subset(0,r^*)$ to a function $y_0$ given by
\begin{align*}
    y_0(r):=\lim_{r_0\to0^+} \psi_{r_0}(r).
\end{align*}
By \eqref{G-ODE-yr}, the derivatives $\psi_{r_0}^\prime$ converge uniformly to the function $r\mapsto G(r,y_0(r))$ in any compact interval $I\subset (0,r^*)$ as $r_0\to0$. Therefore, $y_0$ is differentiable and satisfies
\begin{align}\label{y0Prime}
    y_0^\prime(r) = \lim_{r_0\to0^+} \psi^\prime_{r_0}(r) = \lim_{r_0\to0^+} G(r_0+r,\psi_{r_0}(r)) = G(r,y_0(r)).
\end{align}
Since $y(r)$ satisfy \eqref{y0Prime}, its graph is the trace of an orbit in $(0,+\infty)\times (0,1]$, and we can extend $y_0:(0,+\infty)\to(0,1)$. See Proposition \ref{proposition-orbitTrapped} for $Q_1$. 

From now on we consider $m\leq n$.
By Remark \ref{Sm-behavior}, the following implications hold:
\begin{align*}
m<n\Rightarrow& 
    0<\frac{1-\Gamma}{r^2} = \frac{(c\chi)^{\frac{2}{m\alpha}}\Gamma^{\frac{2}{m\alpha}}}{(1+\Gamma)r^2} 
    \left[ (c\chi)^{-1/\alpha}
    \frac{(1-\Gamma^2)^{m/2}}{\Gamma^{1/\alpha}}
    \right]^{\frac{2}{m}} 
    \\
    &\quad =\frac{(c\chi)^{\frac{2}{m\alpha}}\Gamma^{\frac{2}{m\alpha}}}{(1+\Gamma)r^2} S_m^{-\frac{2}{m}} \to \frac{1}{2} c^{\frac{2}{m\alpha}} {n-1 \choose m}^{-\frac{2}{m}} =: C_m
    \\
m=n \Rightarrow&
    0<\frac{1-\Gamma}{r^{2-2/n}} = \frac{(c\chi)^{\frac{2}{n\alpha}}\Gamma^{\frac{2}{n\alpha}}}{(1+\Gamma)r^{2-2/n}} 
    \left[ (c\chi)^{-1/\alpha}
    \frac{(1-\Gamma^2)^{n/2}}{\Gamma^{1/\alpha}}
    \right]^{\frac{2}{n}} 
    \\
    &\quad=\frac{(c\chi)^{\frac{2}{n\alpha}}\Gamma^{\frac{2}{n\alpha}}}{(1+\Gamma)r^{2-2/n}} S_n^{-\frac{2}{n}} \to \frac{1}{2} c^{\frac{2}{n\alpha}} =: C_n
\end{align*}
as $r\to0$.
We define
\begin{align*}
    z_m(r) =& \frac{1-\Gamma(r)}{r^2} - C_m, \quad m<n
    \\
     z_n(r) =& \frac{1-\Gamma(r)}{r^{2-2/n}} - C_n,
\end{align*}
and we have that $z_m(r)\geq0$ and $\lim_{r\to0} z_m(r)=0$, $m\in\{2,\dots,n\}$. Observe that
\begin{align}
    0<1-y_0(r)<1-\Gamma(r) = 
    \begin{cases}
        r^2(C_m+z_m), \qquad m<n, \\
        r^{2-2/n}(C_n+z_n), \ \, m=n.
    \end{cases}
\end{align}
therefore $\lim_{r\to0}y_0(r) =0$ and, for $m<n$ or $m = n \geq 3$, 
\[
    \lim_{r\to0}\frac{1-y_0(r)}{r} = 0.
\]
For $m=n=2$, by Remark \ref{Sm-behavior} and \eqref{G-ODE-yr} we have $\lim_{r\to0} y_0^\prime(r) = 0$. We conclude that $y_0$ is differentiable at $r=0$ and $y_0^\prime(0) =0$. Now we define
\begin{align*}
    w_m := 1-D_m r^2,
\end{align*}
$D_m>0$ to be determined. Therefore $w_m^\prime = -2D_mr$ and, for $m<n$,
\begin{align*}
    G(r,w_m) =& {n-1 \choose m-1}^{-1} \left(\frac{\xi}{\xi^\prime}\right)^{m-1} (c\chi)^{1/\alpha} \frac{1-w_m^2}{w_m}
    \left[  
    \frac{S_m}{(c\chi)^{1/\alpha}} - \frac{w_m^{1/\alpha}}{(1-w_m^2)^{m/2}} 
    \right] 
    \\
    =& {n-1 \choose m-1}^{-1} \frac{(c\chi)^{1/\alpha}}{(\xi^\prime)^{m-1}}\left(\frac{\xi}{r}\right)^{m-1}  \frac{1+w_m}{w_m} D_m r^{m+1}
    \\
    &\times\left\{  
    \frac{1}{(c\chi)^{1/\alpha}}{n-1 \choose m}r^{-m} + O(r^{-(m-1)})  - \frac{w_m^{1/\alpha}}{(1+w_m)^{m/2}} \frac{1}{(D_mr^2)^{m/2}} 
    \right\}
    \\
    =& {n-1 \choose m-1}^{-1} \frac{(c\chi)^{1/\alpha}}{(\xi^\prime)^{m-1}}\left(\frac{\xi}{r}\right)^{m-1}  \frac{1+w_m}{w_m} D_m
    \\
    &\times\left\{  
    \left[ 
    \frac{1}{(c\chi)^{1/\alpha}}{n-1 \choose m} - \frac{w_m^{1/\alpha}}{(1+w_m)^{m/2}}\frac{1}{D_m^{m/2}}
    \right] r + O(r^{2}) 
    \right\}.
\end{align*}
For $m=n$,
\begin{align*}
    G(r,w_n) =& \left(\frac{\xi}{\xi^\prime}\right)^{n-1} (c\chi)^{1/\alpha} \frac{1-w_n^2}{w_n}
    \left[  
    \frac{S_n}{(c\chi)^{1/\alpha}} - \frac{w_n^{1/\alpha}}{(1-w_n^2)^{m/2}} 
    \right] 
    \\
    =&  \frac{(c\chi)^{1/\alpha}}{(\xi^\prime)^{n-1}}\left(\frac{\xi}{r}\right)^{n-1}  \frac{1+w_n}{w_n} D_n r^{n+1}
    \left\{  
    O(r^{-(n-2)})  - \frac{w_n^{1/\alpha}}{(1+w_n)^{n/2}} \frac{1}{(D_nr^2)^{n/2}} 
    \right\}
    \\
    =& \frac{(c\chi)^{1/\alpha}}{(\xi^\prime)^{n-1}}\left(\frac{\xi}{r}\right)^{n-1}  \frac{1+w_n}{w_n} D_n
    \left\{  
    - \frac{w_m^{1/\alpha}}{(1+w_n)^{m/2}}\frac{1}{D_n^{n/2}}
    r + O(r^{3}) 
    \right\}.
\end{align*}
We conclude that there exists $\epsilon>0$ such that $0\leq r < \epsilon \Rightarrow w'(r) > G(r,w(r))$, respectively $w'(r) < G(r,w(r))$, if and only if

\begin{align*}
    &{n-1 \choose m-1}^{-1} c^{1/\alpha} \cdot 2D_m \left[\frac{1}{c^{1/\alpha}} {n-1 \choose m} -\frac{1}{2^{m/2}}\frac{1}{D_m^{m/2}}\right]   <   -2D_m \Leftrightarrow
    \\
    &(2D_m)^{-m/2} - \frac{1}{c^{1/\alpha}} {n-1 \choose m} > {n-1 \choose m-1} c^{-1/\alpha} \Leftrightarrow
    \\
    &(2D_m)^{-m/2} >\left[{n-1 \choose m-1}+{n-1 \choose m} \right] c^{-1/\alpha} = {n \choose m} c^{-1/\alpha} \Leftrightarrow
    \\
    &D_m < \frac{1}{2}c^{\frac{2}{m\alpha}} {n\choose m}^{-2/m}, \quad 
    \mbox{(respectively) }  D_m > \frac{1}{2}c^{\frac{2}{m\alpha}} {n\choose m}^{-2/m},
\end{align*}
where $m=2,\dots,n$. 

In order to finish the proof, we will need the following technical result.

\medskip

\noindent {\bf Claim.}
 \textit{   Fix $D<\frac{1}{2}c^{\frac{2}{m\alpha}} {n\choose m}^{-2/m}$ (respectively $D>\frac{1}{2}c^{\frac{2}{m\alpha}} {n\choose m}^{-2/m}$) and take $\epsilon>0$ such that $0<r<\epsilon\Rightarrow$ $w'(r)>G(r,w(r))$ (respectively  $w'(r)<G(r,w(r))$). Then, $0<r<\epsilon\Rightarrow$ $w(r)\geq y(r)$ (respectively $w(r)\leq y(r)$).
}
    \medskip

\noindent \textit{Proof of the claim}.
 Fix $D<\frac{1}{2}c^{\frac{2}{m\alpha}} {n\choose m}^{-2/m}$ and $\epsilon>0$ as in the statement of the Claim. Suppose that there exists $0<r_1<\epsilon$ such that $w(r_1) < y(r_1)$. Define $r_2=\inf\{r\in[0,r_1] ; \, t\in(r,r_1) \Rightarrow  w(t)<y(t) \}$. By definition, $w(r_2)=y(r_2)$ and $w'(r_2) < y'(r_2)$. If $r_2=0$, $0=w'(0)<y'(0)=0$, a contradiction. If $r_2>0$, $w'(r_2) > G(r_2,w(r_2))=G(r_2,y(r_2))=y'(r_2)$, a contradiction. The proof is analogous for $D>\frac{1}{2}c^{\frac{2}{m\alpha}} {n\choose m}^{-2/m}$ and the proof of the claim is completed.

\medskip

By the claim above, for all $\delta>0$ there exists $\epsilon>0$ such that $0<r<\epsilon$ implies
\[
    \left[\frac{1}{2}c^{\frac{2}{m\alpha}}{n\choose m}^{-2/m} - \delta  \right] r^2 
    < 1-y_0(r) < \left[\frac{1}{2}c^{\frac{2}{m\alpha}}{n\choose m}^{-2/m}+\delta\right]r^2.
\]
Since $\lim_{r\to0}\phi(r)\to 0$, we may take $\epsilon>0$ such that $r<\epsilon\Rightarrow$ $\phi^2(r) < 6\delta^\prime$ for a $\delta^\prime>0$ fixed. 
A straightforward computation shows that
\begin{align*}
    &\left(1 - \frac{\phi^2(r)}{6}\right) \phi(r)\phi^\prime(r) <-y_0^{\prime}(r) = \sin (\phi(r)) \phi^\prime(r) = \frac{\sin(\phi(r))}{\phi(r)} \phi(r) \phi'(r)  < \phi(r)\phi'(r) \Rightarrow
    \\
    &(1- \delta') \frac{(\phi^2)'(r)}{2}  <  -y_0^{\prime}(r)  < \frac{(\phi^2)'(r)}{2} \Rightarrow
    \\
     &(1- \delta') \frac{\phi^2(r)}{2} < \int_{0}^{r} (-y_0^{\prime}(\rho)) d\rho = 1-y_0(r) < \frac{\phi^2(r)}{2} \Rightarrow
     \\
     &\sqrt{\frac{1-\delta'}{2}} \phi(r) < r\sqrt{\frac{1}{2}c^{\frac{2}{m\alpha}}{n\choose m}^{-2/m} + \delta} 
     \quad \mbox{ and } \quad
    \frac{1}{\sqrt{2}} \phi(r) > r\sqrt{\frac{1}{2}c^{\frac{2}{m\alpha}}{n\choose m}^{-2/m} - \delta}.
\end{align*}
Since $\delta,\delta'$ are arbitrarily small, we conclude by Squeeze Theorem that 
$$
\phi'(0) = \lim_{r\to0} \frac{\phi(r)}{r} = c^{\frac{1}{m\alpha}}{n\choose m}^{-1/m}.
$$
Using \eqref{1/F-Big-ODE}, 
\begin{align*}
    \phi'(r) =& {n-1 \choose m-1}^{-1} \frac{(c\chi)^{1/\alpha}}{(\xi')^{m-1}} \left(\frac{\xi}{r}\right)^{m-1}  \bigg[ \cos^{\frac{1}{\alpha}-1}\phi \left( \frac{r}{\sin\phi} \right)^{m-1} 
    \\
    &\qquad \qquad \qquad \qquad \qquad \quad  - \frac{\tan\phi}{r} \frac{1}{(c\chi)^{1/\alpha}}  {n-1 \choose m} +O(r)
    \bigg] \to c^{\frac{1}{m\alpha}} {n\choose m}^{-1/m}
\end{align*}
as $r\to0$, which proves that the bowl soliton is $C^2$ for either $m<n$ or $m=n$ and $X$ nonparallel, with
\begin{align}
    s'(r) = \frac{1}{\chi(r)} \tan\phi(r) =  rc^{\frac{1}{m\alpha}} {n\choose m}^{-1/m} + O(r^2),
\end{align}
by Structural Condition \ref{StructuralConditionrto0} and $\tan\phi$ expansion near $\phi=0$. Integrating this expression yields \eqref{bowl-expansion}.

To show that $\Sigma_0$ is strictly convex, we use \eqref{kappa-eq} to deduce that
\begin{align*}
    \sin\phi_0 \chi(r) \kappa_\tau = \chi(r) (-\dot y_0) + \chi'(r) (1-y_0^2))>0,
\end{align*}
since $\chi(r)$ is a positive non-decreasing function of $r$. \hfill $\square$

\subsection{Gauss-Kronecker curvature case for parallel vector fields}

In the case when  $X$ is a parallel vector field, the Gauss-Kronecker curvature of the cylinders vanishes, that is, 
\begin{align}
    S_n(r)\equiv 0,
\end{align}
by \eqref{S-nformula-} and from the fact that  $\chi^\prime\equiv0$.
In this case, the ODE system (\ref{ODE-Matrix}) becomes 
\begin{align}
\label{ODESystem-Parallel}
    \begin{pmatrix}
    \dot r \\
    \dot \phi
    \end{pmatrix}
    =
    \begin{pmatrix}
    \cos \phi\\
    \left(\frac{\xi}{\xi'}\right)^{n-1} c^{1/\alpha} \sin\phi\frac{\cos^{\frac{1}{\alpha}}\phi}{\sin^n\phi}
    \end{pmatrix}
\end{align}
a particular case that can be explicitly solved as discussed in the following proposition.

\begin{proposition}
\label{PropXparallelGeneral}
    Fixing $(r_0,\phi_0)$, $\phi \not\equiv \pi/2 \mbox{ (mod } \pi)$, in the phase portrait, the trace of the orbit passing through $(r_0,\phi_0)$ solution of \eqref{ODESystem-Parallel} is the graph of the function
    \begin{align}
    \label{IntegralSolution-Parallel}
    r(\phi)  = 
    \begin{cases}
        I_{r_0}^{-1}\left(c^{-1}\int_{\phi_0}^{\phi} \sin^{n-1} \varphi \, d\varphi   \right) \ \ \mbox{ if } \alpha=1, \\
        I_{r_0}^{-1}\left( c^{-n}\int_{\phi_0}^{\phi} \tan^{n-1} \varphi \, d\varphi   \right) \ \ \mbox{ if } \alpha=1/n,
    \end{cases}
\end{align}
where $\phi$ is in the same quadrant of $\phi_0$ and $I_{r_0}:[0,+\infty)\to [b,+\infty)$ is the function given by
\begin{align*}
    I_{r_0}(r) = \int_{r_0}^{r} \left(\frac{\xi(\rho)}{\xi'(\rho)}\right)^{n-1}  \, d\rho, \ \ \ b=\int^0_{r_0}\left(\frac{\xi(\rho)}{\xi'(\rho)}\right)^{n-1}  \, d\rho \leq 0.
\end{align*}
Therefore $r(\phi)$ is bounded for $\alpha=1$ and unbounded for $\alpha=1/n$ as $\phi\to \pm \pi/2$.
\end{proposition}

\noindent \textit{Proof}. It follows from (\ref{ODESystem-Parallel}) that
\begin{align}
\label{phi'm=n}
    \frac{d\phi}{dr} = \left(\frac{\xi}{\xi'}\right)^{n-1} c^{1/\alpha} \tan\phi\frac{\cos^{\frac{1}{\alpha}}\phi}{\sin^n\phi},
\end{align}
a separable first order differential equation whose solution is given by (\ref{IntegralSolution-Parallel}). \hfill $\square$

\medskip





Now, we summarize the existence results for translating solitons in the case when $X$ is parallel vector field, which corresponds to Riemannian products $\mathbb{P}\times\mathbb{R}$.

\begin{theorem}
\label{ExistenceTheorem-m=nXparallel}
Let $\mathbb{P}^n\times \mathbb{R}$ be the Riemannian product for which  the functions $\xi$ and $\chi=1$ in \eqref{riem-amb-met} satisfy the structural conditions \ref{StructuralConditionChi1} to \ref{StructuralConditionrtoinfty}.  Let $\alpha\in\{1,1/n\}$. Then the  rotationally symmetric Gauss-Kronecker curvature solitons  with respect to the parallel vector field $X$ and $\alpha$ are classified as follows:
    \begin{itemize}
        \item[(a)] Vertical cylinders of the form $\partial B(o,r) \times \mathbb{R}$, for any $r>0$.
        \item[(b)] Bowl solitons, each one  given by a  strictly convex $C^2$ graph  contained in the half-space $\mathbb{P}\times [s_0, +\infty)$, where $s_0$ is the minimum of the function $s(r)$ that defines the graph. Moreover, each bowl soliton is contained in a vertical cylinder for $\alpha=1$, which asymptotes it; for $\alpha=1/n$, it is not contained in any vertical cylinder and $s(r)$ is defined over $\mathbb{P}$. Finally,  $s(r)$ satisfies \eqref{bowl-expansion} as $r\to 0$.
        \item[(c)] In the case when $n$ be even and $\alpha=1/n$,  there exists two one-parameter families $\mathcal{C}_0=\{\Sigma_{r_0}; \ r_0>0\}$ and $\mathcal{C}_1=\{\Sigma_{\phi_0}; \ \phi_0\in Q_1\}$ of properly embedded convex rotationally invariant translating solitons, with nonempty boundary and the following properties:
    \begin{itemize}
        \item[1.] For each $r_0>0$, the soliton $\Sigma_{r_0}$ is an unbounded graph contained in $\mathbb{P}\times [s_0,+\infty)$ where $s_0$ is its minimum height; $\Sigma_{r_0}$ is the graph of a function $s(r)$ defined on the complement of the ball $B(o,r_0)\times \{s_0\}\subset \mathbb{P}_{s_0}$; the graph is $C^2$ except at its boundary points, where it is  tangent to $\mathbb{P}_{s_0}$. Its orientation $N$ satisfies $ \langle X,N \rangle > 0$ everywhere. 
        \item[2.] For each $\phi_0\in Q_1$, the soliton $\Sigma_{\phi_0}$ is a conic graph contained in the half-space $\mathbb{P}\times [s_0,+\infty) \times \mathbb{P}$, where $s_0$ is the minimum of the function $s(r)$ in $\mathbb{P}$ that defines the graph and that satisfies $s'(r)=\tan\phi_0$; this function is strictly convex and $C^2$ at the points outside $\partial \Sigma_{\phi_0} = \partial B(o, r_0)\times \{s_0\}\subset \mathbb{P}_{s_0}$.
    \end{itemize}
        \item[(d)] For $n$ even and $\alpha=1$, there exist four one-parameter families $\mathcal{C}^\pm=\{\Sigma_{r_0}^\pm; \, r_0>0\}$, $\mathcal{C}^+_0=\{\Sigma_{\phi_0}^+; \, \phi_0 \in Q_1\cup Q_3\}$ and $\mathcal{C}^-_0=\{\Sigma_{\phi_0}^-; \, \phi_0 \in Q_2\cup Q_4\}$ of properly embedded convex rotationally invariant translating solitons, with the following properties:
    \begin{itemize}
        \item[1.] For each $r_0>0$, the soliton $\Sigma^{\pm}_{r_0}$ is contained in $\mathbb{P}\times [s_0,+\infty)$ and it is the graph of a function defined on the annulus $B(o, r_1) \backslash B(o, r_0) \subset \mathbb{P}$ for some $r_1>r_0$.
        \item[2.] Along their boundaries, $\Sigma^{\pm}_{r_0}$ is tangent to the horizontal hyperplane $\mathbb{P}_{s_0}$ at $r=r_0$ or $r=r_1$, respectively. Moreover, $\Sigma^+_{r_0}$ and $\Sigma^-_{r_0}$ asymptote the cylinders $\partial B(o, r_1)\times \mathbb{R}$ or $\partial B(o, r_0)\times \mathbb{R}$, respectively.
        \item[3.] For each $\phi_0\in Q_1 \cup Q_3$, the soliton $\Sigma_{\phi_0}^+$ is a conic graph whose profile curve $s(r)$ is defined for $r\in[0,r_1)$, $r_1>r_0$, where $s'(0) = \tan \phi_0$ and $\Sigma_{\phi_0}^+$ asymptotes the cylinder $\partial B(o, r_1) \times \mathbb{R}$.
        \item[4.] For each $\phi_0\in Q_2 \cup Q_4$, the soliton $\Sigma_{\phi_0}^-$ is a conic graph whose profile curve $s(r)$ is defined for $r\in(0,r_1]$, $r_1>r_0$, where $s'(r) = \tan\phi_0$ and $s'(r_1) = 0$. If $\phi_0=\pm\pi/2$, the graph of $s(r)$ asymptotes the line $r=0$. 
    \end{itemize}
        \item[(e)] For $n$  odd and $\alpha=1/n$, there exist two  one-parameter families $\mathcal{C}_1=\{\Sigma^+_{\phi_0}; \ \phi_0 \in Q_1\}$ and $\mathcal{C}_2=\{\Sigma^-_{\phi_0}; \ \phi_0 \in Q_2\}$ of properly embedded convex rotationally invariant  translating solitons, with the following properties:
    \begin{itemize}
        \item[1.] For each $\phi_0 \in Q_1$, the soliton $\Sigma^+_{\phi_0}$ is a conic graph of a function defined on $\mathbb{P}$ whose profile curve $s(r)$ satisfies $s'(0)=\tan \phi_0$ and  $\Sigma^+_{r_0}$ is unbounded and contained in $ \mathbb{P}\times [s_0,+\infty)$.
        \item[2.] For each $\phi_0 \in Q_2$, the soliton $\Sigma^-_{\phi_0}$ is a conic graph of a function defined in $\mathbb{P}$ whose profile curve $s(r)$ satisfies $s'(0)=\tan \phi_0$ and  $\Sigma^-_{r_0}$ is unbounded, $C^2$-singular at its minimum height, and contained in  $\mathbb{P}\times [s_0,+\infty)$.
    \end{itemize}
        \item[(f)] For $n$ odd and $\alpha=1$, there exist three  one-parameter families $\mathcal{C}=\{\Sigma^+_{r_-}; \ r_- \geq 0\}$, $\mathcal{C}_1=\{\Sigma^+_{\phi_0}; \ \phi_0 \in Q_1\}$ and $\mathcal{C}_2=\{\Sigma^-_{\phi_0}; \ \phi_0 \in Q_2\}$ of properly embedded convex and rotationally invariant translating solitons, with the following properties:
    \begin{itemize}
        \item[1.] For each $r_-\geq 0$, the soliton $\Sigma_{r_-}$ is an unbounded graph contained on $(B(o,r_+) \setminus B(0, r_-)) \times [s_0,+\infty)$ for some $r_+>r_0>r_-$. Moreover,  $\Sigma_{r_-}=\Sigma^+_{r_0} \cup \Sigma^-_{r_0}$, where $\Sigma^+_{r_0}\subset (B(o,r_+) \setminus B(0, r_0)) \times [s_0,+\infty)$ is tangent to $\mathbb{P}_{s_0}$ at $\partial\Sigma^{+}_{r_0} = \partial B(o,r_0)\times \{s_0\}\subset \mathbb{P}_{s_0}$ and asymptotes the cylinder $ \partial B(o,r_+)\times \mathbb{R}$, and $\Sigma^-_{r_0}\subset (B(o,r_0) \setminus B(0, r_-)) \times [s_0,+\infty)$ is tangent to $\mathbb{P}_{s_0}$ at $\partial\Sigma^{-}_{r_0} = \partial B(o,r_0)\times \{s_0\}\subset \mathbb{P}_{s_0}$ and asymptotes the cylinder $\partial B(o,r_-)\times \mathbb{R}$
        \item[2.] For each $\phi_0 \in Q_1$, there exists $r_1>0$ such that $\Sigma^+_{\phi_0}$ is a conic unbounded graph contained in $B(o,r_1)\times [s_0,+\infty)$ of a function defined on some ball $B(o,r_1) \subset\mathbb{P}$ whose profile curve $s(r)$ satisfies $s'(0)=\tan \phi_0$ and  $\Sigma^+_{\phi_0}$ asymptotes the cylinder $\partial B(o,r_1)\times \mathbb{R}$.
        \item[3.] For each $\phi_0 \in Q_2$, $\Sigma^-_{\phi_0}$ is a conic graph of a function defined in $B(o,r_1)$, for some $r_1>0$, whose profile curve $s(r)$ satisfies $s'(0)=\tan \phi_0$ and $\Sigma^-_{\phi_0}$ is tangent to $\mathbb{P}_{s_0}$ at  $\partial \Sigma^-_{\phi_0} = B(o,r_1)\times \{s_0\}$. Moreover, $\Sigma^-_{\phi_0}$ is unbounded if and only if $\phi_0= \pm \pi/2$. 
    \end{itemize}
    \end{itemize}
\end{theorem}

\noindent \textit{Proof of Theorem \ref{ExistenceTheorem-m=nXparallel}.} (a) The unique equilibrium points of (\ref{ODESystem-Parallel}) are $(r_0,\pm\pi/2)$, where $r_0$ is a positive constant. These points correspond to vertical cylinders.

\medskip

\noindent (b)
Take the limits $r_0\to0$ and $\phi_0\to0$ in (\ref{IntegralSolution-Parallel}) and we have
\begin{align*}
        r(\phi) =
         \begin{cases}
        I_{0}^{-1}\left(c^{-1}\int_{0}^{\phi} \sin^{n-1} \varphi \, d\varphi   \right) \ \ \mbox{ if } \alpha=1, \\
        I_{0}^{-1}\left( c^{-n}\int_{0}^{\phi} \tan^{n-1} \varphi \, d\varphi   \right) \ \ \mbox{ if } \alpha=1/n.
    \end{cases}.
    \end{align*}
The orbit whose trace is the graph of the above function is the one corresponding to the bowl soliton of the Gaussian curvature. It is strictly convex by \eqref{kappa-eq}. Since $\int_{0}^{\pi/2} \sin^{n-1}\phi \, d\phi$ converges and $\int_{0}^{\pi/2} \tan^{n-1}\phi \, d\phi$ diverges, the bowl soliton is contained in a cylinder, if $\alpha=1$, or it is a graph on $\mathbb{P}$, if $\alpha=1/n$.
A straightforward computation shows that, if $H(\phi) = \int_{0}^{\phi} \sin^{n-1}\varphi d\varphi$, for $\alpha=1$, then
\begin{align*}
    H(0) = H'(0) = \dots= H^{(n-1)}(0)=0, H^{(n)}(0) = (n-1)! 
\end{align*}
what implies that
\begin{align*}
    \frac{r^n}{n} + O(r^{n+1}) = I_0(r) = \frac{1}{c}H(\phi) = \frac{1}{c}\left(\frac{H^{(n)}(0)}{n!}\phi^n + O(\phi^{n+1})\right) =\frac{1}{nc}\phi^{n}+O(\phi^{n+1})
\end{align*}
By \eqref{IntegralSolution-Parallel}, $\lim_{r\to0} \phi(r) =0$, therefore for all $\delta>0$ there exists $\epsilon>0$ such that $0 < r < \delta$ yields
\begin{align*}
    \frac{\phi^n}{r^n}(1-\epsilon) < c +O(r) < \frac{\phi^n}{r^n}(1+\epsilon)
\end{align*}
Finally, $\alpha=1 \Rightarrow$ $\phi'(0) = c^{1/n}$. For $\alpha=1/n$, $\phi'(0) = c$ and the proof is essentially the same. Finally, as $r\to0$ we have \eqref{bowl-expansion}.

\medskip

\noindent (c) Since $\cos\phi>0$, we consider $\phi\in Q_1 \cup Q_4$.
Let us consider an orbit in $\mathbb{R}_+ \times Q_1$ with an endpoint at $(r_0,0)$. By Proposition \ref{PropXparallelGeneral}, this orbit is a graph of a function $\phi(r)$ for $r \geq r_0$. This orbit corresponds to $\Sigma_{r_0}$. An orbit with an endpoint $(0,\phi_0)$, $\phi_0\in(0,\pi/2)$, is again the graph of a function $\phi(r)$ for $r\geq 0$. This orbit corresponds to $\Sigma_{\phi_0}$. It is unbounded, since 
\begin{align*}
    \frac{ds}{dr}= \tan\phi\Rightarrow s(r)-s(r_1)=\int_{r_1}^r \tan\phi \, d\rho > C \int_{r_1}^{r}d\rho \underset{r\to+\infty}{\longrightarrow}+\infty
\end{align*}

\medskip

\noindent (d)
Let us consider an orbit in $\mathbb{R}_+ \times Q_1$ with an endpoint at $(r_0,0)$. By Proposition \ref{PropXparallelGeneral}, this orbit is a graph of a function $\phi(r)$ for $r_0\leq r\leq r_1$ with the other endpoint at $(r_1,\pi/2)$. This orbit corresponds to $\Sigma_{r_0}^+$. An orbit with an endpoint $(0,\phi_0)$, $\phi_0\in(\pi/2,\pi)$, has the other endpoint at some $(r_1,\pi)$, where $r_1(\phi_0)$ is a continuous decreasing function. See Remark \ref{rmk-continuityofendpoints-orbits}. Therefore, $r_n\doteq\lim_{\phi_0\to\pi/2} r_1(\phi_0)$ is such that an orbit with an endpoint at $(r_1,\pi)$, $r_1>r_n$, has the other endpoint at $(r_0,\pi/2)$, $0<r_0<r_1$. This orbit corresponds to $\Sigma_{r_0}^{-}$. An orbit with an endpoint at $(r_1,\pi)$, $r_1<r_n$, has another endpoint at some $(0,\phi_0)$, $\phi_0\in Q_2$. This orbit corresponds to $\Sigma^-_{\phi_0}$. The other orbit with an endpoint at $(0,\phi_0)$, $\phi_0\in Q_1$, corresponds to $\Sigma^+_{\phi_0}$.

Now we show that $s(r)$ is unbounded. We fix $\phi\in Q_1$. The other cases are similar. Let $r_0<r_-<r \leq r_1$. Then $r\to r_1$ $\Leftrightarrow$ $\phi\to \pi/2$ and
\begin{align*}
&     s(r) - s(r_-) = \int_{r_-}^{r} \frac{ds}{dr} d\rho 
    \\
    & = \int_{r_-}^{r}\tan\varphi \, d\rho 
    = \frac{1}{c}\int_{\phi(r_-)}^{\phi(r)}\left(\frac{\xi^\prime(\rho)}{\xi(\rho)}\right)^{n-1} \frac{\sin^n\varphi}{ \cos \varphi} \, d\varphi\geq \int_{\phi(r_-)}^{\phi} \Lambda \frac{\sin^n\varphi}{\cos\varphi}\, d\varphi \underset{\phi\to\pi/2}{\longrightarrow} +\infty,
\end{align*}
where
\begin{align*}
    \Lambda = \min_{r\in[r_0,r_1]} \left\{ \frac{1}{c} \left(\frac{\xi^\prime(r)}{\xi(r)}\right)^{n-1} \right\}.
\end{align*}

\noindent (e)
We claim that an orbit in $\mathbb{R}_+ \times Q_2$ or $\mathbb{R}_+ \times Q_4$ must have an endpoint at some $(0,\phi_0)$, where $\phi_0\in Q_2$ or $\phi_0\in Q_4$, respectively. Suppose that the orbit in $\mathbb{R}_+\times Q_4$ has an endpoint at $(r_-,-\pi/2)$ for some $r_->0$. Observe that, for $r>r_-$, we have by \eqref{phi'm=n} that
$$
    +\infty =\left| \int_{\phi}^{-\pi/2} |\tan \varphi|^{n-1} \, d\varphi \right|= c^n \left|\int_{r}^{r_-} \left(\frac{\xi(\rho)}{\xi^\prime(\rho)}\right)^{n-1}\, d\rho \right|,
$$
which is a contradiction. The orbits in $\mathbb{R}_+ \times Q_1$ and $\mathbb{R}_+ \times Q_4$ with same endpoint $(r_0,0)$ match orientation and their union corresponds to $\Sigma_{\phi_0}^-$, where $\phi_0\in Q_4$ and $(0,\phi_0)$ is the other endpoint of the orbit in $\mathbb{R}_+\times Q_4$. An orbit with an endpoint at $(0,\phi_0)\in \mathbb{R}_+\times Q_1$ is contained in $\mathbb{R}_+\times Q_1$ by Proposition \ref{PropXparallelGeneral} and corresponds to $\Sigma_{\phi_0}^+$. Finally, both $\Sigma^-_{\phi_0}$ and $\Sigma^+_{\phi_0}$ are unbounded by Proposition \ref{PropXparallelGeneral}.

\medskip

\noindent (f) 
An orbit in $\mathbb{R}_+\times Q_4$ with an endpoint at $(r_-,-\pi/2)$, $r_->0$, has the other endpoint at some $(r_0,0)$, $r_+>r_0$, by Proposition \ref{PropXparallelGeneral}. This orbit corresponds to $\Sigma^-_{r_0}$. Now the orbit in $\mathbb{R}_+\times Q_1$ with an endpoint at $(r_0,0)$ has an endpoint at $(r_+,\pi/2)$, $r_-<r_+<r_+$. This orbit corresponds to $\Sigma^+_{r_0}$ and these orbits match orientation at $(r_0,0)$. An orbit with an endpoint at $(0, \phi_0)$, $\phi_0\in Q_1$, converges to some $(r_1,\pi/2)$, $r_1>0$. This orbit corresponds to $\Sigma_{\phi_0}^{+}$. The analogous in $\mathbb{R}_+\times Q_2$ gives $\Sigma^{-}_{\phi_0}$. This finishes the proof of the theorem. \hfill $\square$

\begin{proposition}
    The only rotationally symmetric translating solitons with respect to $X$ are open sets of the solutions described by Theorems \ref{ExistenceTheorem-m<nXnonparallel} and \ref{ExistenceTheorem-m=nXparallel}, up to reflection \eqref{reflection-symmetry}, if $m$ is even, or reversion of orientation \eqref{reversion-of-orientation}, if $m$ is odd.
\end{proposition}
\noindent \textit{Proof.} Each rotationally symmetric translating soliton corresponds to a subset of an orbit and all the possible classes of orbits are fully described in the statements of Theorems \ref{ExistenceTheorem-m<nXnonparallel} and \ref{ExistenceTheorem-m=nXparallel}. \hfill $\square$

\begin{remark}
    In \cite{Ronaldo}, de Lima and Pipoli consider translators in $\mathbb{R}^{n}\times \mathbb{R}$ and $\mathbb{H}^n\times \mathbb{R}$ for $\alpha=1$. In their Theorem 1, they obtain families of rotationally symmetric translators $\mathcal{C}_r$, for $r<n$ odd, and $\mathcal{C}_r^1$, $\mathcal{C}_r^2$, for $r<n$ even, which correspond to $\mathcal{C}^3$, $\mathcal{C}^1$ and $\mathcal{C}^4$, respectively, obtained in our Theorem \ref{ExistenceTheorem-m<nXnonparallel}. Their methods are different from ours, and their translators in $\mathcal{C}^2_r$ are defined only for $\langle X,N \rangle \geq0$, being zero at its boundary, where it is orthogonal to horizontal plane. Our translators in $\mathcal{C}^4$ have a neck, where $\langle X,N \rangle = 0$ and it is orthogonal to the horizontal plane, and contains a region where $\langle X,N \rangle < 0$, being $-1$ at its boundary, where it is tangent to the horizontal plane. See Figure \ref{fig:profile-curves-m-even}.
\end{remark}

\subsection{Asymptotic limits}\label{asymptotic}

In this section, we describe the asymptotic behavior of translating solitons in terms of the  function
\begin{equation}
    \label{Gamma}
    \Gamma(r) \doteq \cos \Phi_{m, \alpha}(r),
\end{equation}
where the function $\Phi_{m, \alpha}$ is defined in \eqref{f-Phima}.

\begin{lemma}
\label{lemma-Gammabehavior}
     For either $m<n$ or $m=n$ and $X$ nonparallel, $\Gamma(r)$ is a nonincreasing positive function that converges to some $\Gamma_\infty\in[0,1)$. For $m<n$, it holds that
     \begin{itemize}
         \item 
     If the metric is asymptotic to the Euclidean model $\mathbb{R}^{n+1}$, then $\Gamma(r)= r^{-m\alpha}/(c\chi_\infty) + O(r^{-(m\alpha+1)})$; 
     \item If the metric is asymptotic to the Riemannian product $\mathbb{H}^n\times\mathbb{R}$ model, then $\Gamma(r) \searrow \Gamma_\infty$, where $\Gamma_\infty\in(0,1)$; and
     \item If the metric is asymptotic to the Hyperbolic model $\mathbb{H}^{n+1}$, then $\Gamma(r) =  e^{-r}[S_\infty^{\alpha}/c+O(r^{-1})]$. 
          \end{itemize}
     Moreover, if $m=n$ and the metric is asymptotic to the hyperbolic model, we also conclude that $\Gamma(r) = e^{-r}[S_\infty^{\alpha}/c+O(r^{-1})]$.
\end{lemma}

\noindent  \textit{Proof}. 
By Remark \ref{Sm-behavior}, the function $S_m(r)/(c\chi)^{1/\alpha}$ is a nonincreasing positive function, therefore it converges to zero or to a positive number as $r\to+\infty$. Since $f(\Gamma)=S_m(r)/(c\chi)^{1/\alpha}$ and $f$ is an increasing function, we conclude that $\Gamma(r)\searrow\Gamma_\infty\in[0,1)$. 

Again by Remark \ref{Sm-behavior}, for a metric asymptotic to the Euclidean model $\mathbb{R}^{n+1}$ and $m<n$, $f(\Gamma(r))=S_m(r)/(c\chi)^{1/\alpha}= r^{-m}/(c\chi_\infty)^{1/\alpha} + O(r^{-(m+1)})\searrow0$. Since $f(x)\approx x^{\frac{1}{\alpha}}$ for $x$ sufficiently closed to 0, $\Gamma(r)= r^{-m\alpha}/(c\chi_\infty) + O(r^{-(m\alpha+1)})$ as $r\to+\infty$. 
 For a metric asymptotic to the Riemannian product $\mathbb{H}^{n}\times\mathbb{R}$ model and $m<n$, $f(\Gamma(r))=S_m(r)/(c\chi(r))^{1/\alpha}\searrow S_\infty>0\Rightarrow \Gamma(r)\searrow\Gamma_\infty \in(0,1)$ .
For a metric asymptotic to the Hyperbolic model $\mathbb{H}^{n+1}$ and $m\leq n$, $f(\Gamma(r))=S_m(r)/(c\chi(r))^{1/\alpha} =S_\infty/(c e^{r})^{1/\alpha} + O(e^{-r/\alpha}/r) \searrow0 \Rightarrow \Gamma(r)=e^{-r}[S_\infty^{\alpha}/c+O(r^{-1})]\searrow 0$. \hfill $\square$

\begin{proposition}
    Let $h(r)$ be a positive function, $h\geq1$, such that
    \begin{align*}
        \frac{\Gamma(r)(1-\Gamma^2(r))}{1+\Gamma^2(r)(m\alpha-1)}\left[\log \left(\frac{\chi^{1/\alpha}(r)}{S(r)}\right)\right]^\prime = o\left(\frac{1}{h(r)}\right) \ \mbox{ and } \
        \frac{h'(r)}{h(r)}\to0
    \end{align*}
    as $r\to+\infty$. Then, for $\alpha\in\{1/m,1\}$ and either $m\in\{2,\dots, n-1\}$ or $m=n$, $X$ nonparallel and the metric asymptotic to the hyperbolic model, 
    \begin{align*}
        y(r) = \Gamma(r) + o\left(\frac{1}{h(r)}\right).
    \end{align*}
    For a metric asymptotic to $\mathbb{R}^{n+1}$, the function $h(r)=r^{(m+\delta)\alpha}$, for $\delta\in(0,1)$, satisfies the hypotheses above. For a metric asymptotic to $\mathbb{H}^n\times\mathbb{R}$ or to $\mathbb{H}^{n+1}$, the function $h(r)=e^{r^\delta}$, for $\delta\in(0,1)$, satisfy the hypotheses above.
\end{proposition}
\noindent \textit{Proof}.
Let us fix an orbit in $\mathbb{R}_+\times (0,1]$ for \eqref{G-ODE-yr}. By Lemma \ref{lemma-Gammabehavior}, there exist $R,\epsilon>0$ such that $r>R$ yields $0\leq y < 1-\epsilon$. We define
\begin{align*}
    \psi := y-\Gamma.
\end{align*}
A straightforward computation shows that
\begin{align*}
    \psi' = \frac{1}{{n-1 \choose m-1}} \left(\frac{\xi}{\xi'}\right)^{m-1} (c\chi)^{1/\alpha}\frac{1-y^2}{y} [f(\Gamma)-f(y)]
    +\alpha\frac{\Gamma(1-\Gamma^2)}{1+\Gamma^2(m\alpha-1)} \left[\log \left( \frac{S_m(r)}{\chi^{1/\alpha}} \right)\right]'.
\end{align*}
Suppose that $\underset{r\to+\infty}{\limsup}\, \psi(r) = \delta$ for some $\delta>0$. Then, for all $r_1>0$ there exists $r_*>r_1$ such that
$\psi(r_*)>\delta.$

For $\alpha=1$, $f'(x)$ attains its minimum at some $x_0\in[0,1)$ and $f'(x_0)>0$, since $f\in C^1$, $\underset{x\to0}{\lim}\,f'(0)=1$, $f'(x)>0$ for $x\in(0,1)$ and $\underset{x\to1}{\lim}f'(x)=+\infty$. Therefore, for $x_1<x_2$ and $x_2-x_1>\delta$,
\begin{align*}
    f(x_2)-f(x_1)=\int_{x_1}^{x_2}f'(x)dx  \geq f'(x_0)(x_2-x_1)>f'(x_0)\delta.
\end{align*}
For $\alpha=1/m$, since $f(x)$ is a strictly increasing convex function in $[0,1)$, $x\mapsto f(x+\delta)-f(x-\delta)$ is a positive increasing function wherever it is defined. Therefore,
\begin{align*}
    f(x_2)-f(x_1)\geq f(x_2-x_1)-f(0)>f(\delta)>0.
\end{align*}
In any case, there exists $\delta'>0$ such that 
\begin{align*}
    y(r_*)-\Gamma(r_*) >\delta \Rightarrow f(y(r_*))-f(\Gamma(r_*)) >\delta'.
\end{align*}
Therefore
\begin{align*}
    \psi'(r_*) <& -\frac{\delta'}{{n-1 \choose m-1}} \left(\frac{\xi(r_*)}{\xi'(r_*)}\right)^{m-1} (c\chi(r_*))^{1/\alpha}\frac{1-y^2(r_*)}{y(r_*)} \\
    &+\alpha\frac{\Gamma(r_*)(1-\Gamma^2(r_*))}{1+\Gamma^2(r_*)(m\alpha-1)}\left[\log \left( \frac{S_m(r)}{\chi^{1/\alpha}} \right)\right]'_{r=r_*} < -\delta'' <0
\end{align*}
for some $\delta''>0$ and $r_*>r_2>r_1$, $r_2$ sufficiently large, since $\left(\xi/\xi'\right)^{m-1} \chi^{1/\alpha}(1-y^2)/y$ is bounded from below and 
\begin{align*}
    \frac{\Gamma(1-\Gamma^2)}{1+\Gamma^2(m\alpha-1)}\left[\log \left( \frac{S_m(r)}{\chi^{1/\alpha}} \right)\right]'\to0
\end{align*}
as $r\to+\infty$. Define
\begin{align*}
    r_3 := \sup \{r>r_* | \, \psi(t) >\delta \ \forall \ t\in(r_*,r)  \}.
\end{align*}
Then
\begin{align*}
    \psi(r_3) =\delta \ \ \  \psi'(r_3) \leq -\delta''.
\end{align*}
Define
\begin{align*}
    r_4 := \sup \{r>r_3 | \, \psi(t) <\delta \ \forall \ t\in(r_3,r)  \}.
\end{align*}
since $\limsup \psi(r) > \delta$, $r_4$ is finite. Then
\begin{align*}
    \psi(r_4) = \delta \Rightarrow \psi'(r_4) \leq -\delta'',
\end{align*}
which implies $\psi(r)>\delta$ for $r<r_4$ sufficiently closed to $r_4$, which is a contradiction. Analogously, we prove that $\liminf_{r\to + \infty} \psi(r) = 0$ and we conclude that
\begin{align*}
    \lim_{r\to+\infty} \psi(r) = 0.
\end{align*}
Now we estimate the rate of convergence $\psi\to0$. Let $h(r)$ be a function as in the statement of the theorem and set
\begin{align*}
    \lambda = h\psi.
\end{align*}
Therefore
\begin{align*}
    \lambda' =& \frac{1}{{n-1 \choose m-1}} \left(\frac{\xi}{\xi'}\right)^{m-1} (c\chi)^{1/\alpha}\frac{1-y^2}{y} [f(\Gamma)-f(y)]h \\
    &+\alpha\frac{\Gamma(1-\Gamma^2)}{1+\Gamma^2(m\alpha-1)}\left[\log \left( \frac{S_m(r)}{\chi^{1/\alpha}} \right)\right]'h
    + \frac{h'}{h} h \psi. \\
    =&\left[ {n-1 \choose m-1}^{-1} \left(\frac{\xi}{\xi'}\right)^{m-1} (c\chi)^{1/\alpha} \frac{1-y^2}{y}  \frac{f(\Gamma) - f(y)}{\Gamma - y} 
    - \frac{h^\prime}{h}\right] (-\lambda) 
    \\
    &+\alpha\frac{\Gamma(1-\Gamma^2)}{1+\Gamma^2(m\alpha-1)}\left[\log \left( \frac{S_m(r)}{\chi^{1/\alpha}} \right)\right]'h.
\end{align*}
Suppose that $\limsup \lambda(r) > \delta >0$. Then for all $r_1> 0$ there exists $r_*>r_1$ such that
\[
    \lambda(r_*) > \delta.
\]
Then $y>\Gamma$ in $r=r_*$. For $\alpha=1$,
\[
    f(y) - f(\Gamma) > f'(x_0)(y-\Gamma),
\]
where $f'(x_0)>0$ is the minimum of $f'$. For $\alpha=1/m$, since $f(x)$ is an increasing convex function, we have for $x_1<x_2$ that
\[
    f'(x_1) < \frac{f(x_2)-f(x_1)}{x_2-x_1}< f'(x_2).
\]
Therefore
\[
    \frac{f(y)-f(\Gamma)}{y-\Gamma} > f'(\Gamma) = \frac{m}{\Gamma(1-\Gamma^2)}f(\Gamma) \ \ \ \mbox{for } \alpha=1/m.
\]
If the metric asymptotes $\mathbb{H}^n\times \mathbb{R}$ model, $f'(\Gamma)$ converges to a positive constant. If the metric asymptotes the $\mathbb{R}^{n+1}$ model, $\Gamma= O(r^{-1})$ as $r\to + \infty$, therefore $\lim_{r\to 0 } y(r) = 0$ and
\begin{align*}
    -\left(\frac{\xi}{\xi'}\right)^{m-1} (c\chi)^{m} \frac{1-y^2}{y}  \frac{f(\Gamma) - f(y)}{\Gamma - y} <-\left(\frac{\xi}{\xi'}\right)^{m-1}  (c\chi)^{m} \frac{1-y^2}{y}  f'(\Gamma) 
    \to -\infty.
\end{align*}
If the metric asymptotes $\mathbb{H}^{n+1}$ model,
\begin{align*}
    -\left(\frac{\xi}{\xi'}\right)^{m-1} (c\chi)^{m} \frac{1-y^2}{y}  \frac{f(\Gamma) - f(y)}{\Gamma - y} < -\left(\frac{\xi}{\xi'}\right)^{m-1} (c\chi)^{m} \frac{1-y^2}{y}  f'(\Gamma) \\
    = -|O(e^{2r})|\to-\infty.
\end{align*}
In any case, there exists $\delta'>0$ such that
\[
    \lambda(r_*) < - \delta' <0,
\]
which leads to a contradiction. Analogously, $\liminf \lambda(r)< - \delta <0$ leads to a contradiction and the proof is complete. \hfill $\square$

\section{Some explicit solutions in terms of integrals}
\label{Section5}

In the particular case of scalar flow curvature solitons (that is, $m=2$), it is possible to obtain explicit graphical parameterizations of these solitons in terms of integrals involving the geometric data $\xi$ and $\chi$.

\begin{proposition}
\label{PropAnaliticalpha2}
    There is a one-parameter family of rotationally symmetric 2-curvature flow solitons with respect to a Killing vector field $X$  for $\alpha=1/2$ and $n\geq 3$ whose profile curves are given by
    \begin{align}
    \label{Analitic-alpha2nGeral}
        s(r)= \int_{r_0}^{r} \sqrt{\frac{\xi^{n-2}G(\lambda)}{C_0+\int_{0}^{\lambda} (\chi^2(\rho)\xi^{n-2}(\rho))^\prime G(\rho) \, d\rho} -\frac{1}{\chi^2(\lambda)}} \, d\lambda,
    \end{align}
    where $C_0\in\mathbb{R}$, $r_0\geq0$ is the minimum radius which the integrand is defined, $r\geq r_0$ and
    \begin{equation}
    \label{Gn}
        G(r) = \exp \bigg(\frac{2c^2}{n-1}\int^r_0 \chi^2(\varrho)\frac{\xi(\varrho)}{\xi'(\varrho)}d\varrho\bigg). 
    \end{equation}
    The bowl soliton is given when $C=0$ and $r_0=0$, whose profile curve is the graph of \eqref{Analytic-Bowl-m2alpha2}. For $C_0<0$,
    \begin{align}
    \label{Analitic-alpha2nC2}
        s(r) = \int_{r_0}^{r} \sqrt{\frac{\xi^{n-2}G(\lambda)}{\int_{r_0}^{\lambda} (\chi^2(\rho)\xi^{n-2}(\rho))^\prime G(\rho) \, d\rho} -\frac{1}{\chi^2(\lambda)}} \, d\lambda, \ \  r\geq r_0,
    \end{align}
which correspond to the profile curves of rotational translating solitons in the family $\mathcal{C}^2$.
For $C_0>0$,
\begin{align}
\label{Analitic-alpha2nC1}
    s(r) = \int_{r_0}^{r} \sqrt{\frac{\xi^{n-2}G(\lambda)}{\chi^2(r_0)\xi^{n-2}(r_0)G(r_0)+\int_{r_0}^{\lambda} (\chi^2(\rho)\xi^{n-2}(\rho))^\prime G(\rho) \, d\rho} -\frac{1}{\chi^2(\lambda)}} \, d\lambda, \ \  r\geq r_0,
\end{align}
which correspond to the profile curves of rotational translators in the family $\mathcal{C}^1$.
\end{proposition}

\noindent \textit{Proof}. Expression \eqref{G-ODE-yr} yields
\begin{align}
    y'=\frac{1}{n-1}\frac{\xi}{\xi'}(c\chi)^2\frac{1-y^2}{y}\left[ \frac{S_2(r)}{(c\chi)^2} - \frac{y^2}{1-y^2} \right],
    \end{align}
    what implies that
\begin{align}
    (y^2)'+\frac{2}{n-1}\frac{\xi}{\xi'}\left[S_2(r)+(c\chi)^2\right]y^2
    = \frac{2}{n-1}\frac{\xi}{\xi'}S_2(r),
    \label{ODE-Linear}
\end{align}
a linear first order ODE in $y^2$. Solving it and using
\begin{align*}
    y^2=\cos^2\phi=\frac{1}{1+\tan^2\phi}=\frac{1}{1+(\chi s')^2},
\end{align*}
we obtain expression (\ref{Analitic-alpha2nGeral}), where $C_0\in\mathbb{R}$ is a constant of integration, $r_0\geq0$ is the minimum radius where the integrand is defined and $\chi^2(r)\xi^{n-2}(r)G(r)$ is an integrating factor of \eqref{ODE-Linear}, the function $G(r)$ given by \eqref{Gn}.

Now we split our analysis into three cases, according to the sign of $C_0$. The case $C_0=0$ corresponds to bowl soliton, where $r_0=0$. Replacing into (\ref{Analitic-alpha2nGeral}), one gets
\begin{align}
\label{Analytic-Bowl-m2alpha2}
    s(r) = \pm\int_{0}^{r} \sqrt{\frac{\xi^{n-2}G(\lambda)}{\int_{0}^{\lambda} (\chi^2(\rho)\xi^{n-2}(\rho))^\prime G(\rho) \, d\rho} -\frac{1}{\chi^2(\lambda)}} \, d\lambda, \ \ r\geq 0.
\end{align}
For $C_0<0$ we define $r_0>0$ implicitly by
\begin{align*}
    C_0=-\int_{0}^{r_0} (\chi^2(\rho)\xi^{n-2}(\rho))'G(\rho) \, d\rho.
\end{align*}
Replacing this expression into \eqref{Analitic-alpha2nGeral}, one obtains  \eqref{Analitic-alpha2nC2}.
Finally, $C_0>0$ one defines $r_0>0$ implicitly by
\begin{align*}
    C_0=\chi^2(r_0)\xi^{n-2}(r_0)G(r_0) -\int_{0}^{r_0} (\chi^2(\rho)\xi^{n-2}(\rho))'G(\rho) \, d\rho.
\end{align*}
Replacing this expression  into \eqref{Analitic-alpha2nGeral}, one gets  \eqref{Analitic-alpha2nC1}. \hfill $\square$

\medskip

In the particular case when $m=n=2$, the solitons described above are translating solitons for the Gaussian curvature of surfaces. The previous proposition can be rewritten in this case as in the following propositions.

\begin{proposition}
    There is a one-parameter family of rotationally symmetric Gaussian curvature flow solitons with respect to a nonparallel $X$ vector field for $\alpha=1/2$ and $n = 2$ whose profile curves are given by
    \begin{align}
    \label{Analitic-alpha2n2Nonparallel}
        s(r)= \int_{r_0}^{r} \sqrt{\frac{G(\lambda)}{C_0+\int_{0}^{\lambda} (\chi^2(\rho))^\prime G(\rho) \, d\rho} -\frac{1}{\chi^2(\lambda)}} \, d\lambda,
    \end{align}
    where $C_0\in\mathbb{R}$, $r_0\geq0$ is the minimum radius which the integrand is defined, $r\geq r_0$ and $G(r)$ is given by \eqref{Gn}.
    For $C_0<0$,
    \begin{align}
    \label{Analitic-alpha2nC2-n2Nonparallel}
        s(r) = \int_{r_0}^{r} \sqrt{\frac{G(\lambda)}{\int_{r_0}^{\lambda} (\chi^2(\rho))^\prime G(\rho) \, d\rho} -\frac{1}{\chi^2(\lambda)}} \, d\lambda, \ \  r\geq r_0,
    \end{align}
which correspond to the profile curves of rotational translators in $\mathcal{C}^2$.
For $C_0>1$,
\begin{align}
\label{Analitic-alpha2nC1-n2Nonparallel}
    s(r) = \int_{r_0}^{r} \sqrt{\frac{G(\lambda)}{\chi^2(r_0)G(r_0)+\int_{r_0}^{\lambda} (\chi^2(\rho))^\prime G(\rho) \, d\rho} -\frac{1}{\chi^2(\lambda)}} \, d\lambda, \ \  r\geq r_0,
\end{align}
which correspond to the profile curves of rotational translators in $\mathcal{C}^1$.
For $0\leq C_0 \leq 1$,
\begin{align}
\label{Analitic-alpha2nC1-n2NonparallelBowl}
    s(r) = \int_{0}^{r} \sqrt{\frac{G(\lambda)}{\cos^2\phi_0+\int_{0}^{\lambda} (\chi^2(\rho))^\prime G(\rho) \, d\rho} -\frac{1}{\chi^2(\lambda)}} \, d\lambda, \ \  r\geq 0,
\end{align}
where $\phi_0\in[0,\pi/2]$ is the angle that the graph hits $r$-axis at $(r,s)=(0,0)$. The bowl soliton corresponds to \eqref{Analitic-alpha2nC1-n2NonparallelBowl} when $\phi_0=0$.
\end{proposition}

\noindent \textit{Proof}. The proof is the same as in Proposition \ref{PropAnaliticalpha2}.

\hfill $\square$

\begin{proposition}
    There is a one-parameter family of rotationally symmetric translating solitons for the Gaussian curvature with respect to a parallel vector field $X$ for $\alpha=1/2$ and $n=2$, whose profile curves are given by
    \begin{align}
    \label{G2-+}
        s(r) =  \int_{r_0}^{r} \sqrt{\frac{G(\rho)}{G(r_0)}-1} \, d\rho, \ \ \ r\geq r_0, \\
    \label{G2-Bowl}
        s(r) =  \int_{0}^{r} \sqrt{G(\rho)-1} \, d\rho, \ \ \ r\geq 0,\\
    \label{G2-Menos}
        s(r) =  \int_{0}^{r} \sqrt{\frac{G(\rho)}{\cos^2\phi_0}-1} \, d\rho, \ \ \ r\geq 0,
    \end{align}
    where $G(r)$ is given by \eqref{Gn} for $n=2$, \eqref{G2-+} is tangent to $\mathbb{P}_0$ and $r_0\in\mathbb{R}_+$, \eqref{G2-Bowl} is the bowl soliton and \eqref{G2-Menos} corresponds to a conic surface with an angle $\phi_0 \in (0,\pi/2)$ with respect to $r$-axis at $r=0$.
\end{proposition}

\noindent \textit{Proof}. It is a straightforward computation to show \eqref{G2-+}, \eqref{G2-Bowl} and \eqref{G2-Menos} as particular cases of \eqref{ODE-Linear} for which $\chi\equiv1$ and $S_2\equiv0$. \hfill $\square$

\medskip

Still concerning Gaussian curvature solitons (that is, $m = n = 2$) for $\alpha=1$ and a parallel vector field $X$, we can exhibit the right-hand side of equation (\ref{IntegralSolution-Parallel}) in terms of $\tan\phi$ what makes possible to exhibit an explicit  integral expression for the profile curve $s(r)$.

\begin{proposition}
    Non-cylindrical,  rotationally symmetric  Gaussian curvature flow soliton with respect to a parallel $X$ vector field and  $\alpha=1$ are graphs of 
    \begin{align}
    \label{Analitic-alpha1n2}
        s(r) = \int_{r_0}^{r}  \sqrt{\left[c\int_{r_0}^{\rho} \frac{\xi(\lambda)}{\xi^\prime(\lambda)} \, d\lambda -\cos\phi_0\right]^{-2}-1} \, d\rho,
    \end{align}
    where $r_0>0$, $\tan\phi_0=s^\prime(r_0)$ and either $r\in[r_0,r_2)$ or $r\in(\max\{0,r_2\},r_0]$, depending on $\phi_0\in Q_1 \cup Q_4$ or $\phi_0\in Q_2 \cup Q_3$, respectively, where $r_2$ is defined implicitly by 
    \[
        c\int_{r_0}^{r_2} \frac{\xi(\lambda)}{\xi^\prime(\lambda)} \, d\lambda =\cos\phi_0.
    \]
\end{proposition}

\noindent \textit{Proof}. By (\ref{IntegralSolution-Parallel}) and using $\sec^2x = 1 + \tan^2x$ identity, we have 
\begin{align*}
    -\frac{1}{c}(\cos\phi - \cos\phi_0) = I_{r_0}(r) = \int_{r_0}^{r} \frac{\xi(\rho)}{\xi^\prime(\rho)} \, d\rho.
    \end{align*}
    Therefore
    \begin{align*}
    \frac{ds}{dr} = \tan\phi = \pm\sqrt{\left[c\int_{r_0}^{r}\frac{\xi(\rho)}{\xi^\prime(\rho)} \, d\rho - \cos\phi_0 \right]^{-2}-1}.
\end{align*}
Integrating this last expression yields (\ref{Analitic-alpha1n2}).  \hfill $\square$



\bibliography{sn-bibliography.bib}

\end{document}